%% file: main.tex
\documentclass[a4paper]{article}

\usepackage[utf8]{inputenc}

\usepackage{lmodern}

\usepackage{amsmath}        
\usepackage{amsfonts}       
\usepackage{amssymb}        
\usepackage{amsthm}         
\usepackage{bbding}         
\usepackage{bm}             
\usepackage{graphicx}       
\usepackage{fancyvrb}       

\usepackage[style=numeric, maxnames=3, giveninits=true, backend=biber]{biblatex}
\usepackage[nottoc]{tocbibind} 
\usepackage{dcolumn}        
\usepackage{booktabs}       
\usepackage{paralist}       
\usepackage{xcolor}         
\usepackage{mathtools}      
\usepackage{textgreek}      
\usepackage{algorithmic}
\usepackage{algorithm}
\numberwithin{algorithm}{section}

\usepackage{setspace}       
\usepackage{caption}        
\usepackage{subcaption}     
\usepackage{afterpage}

\usepackage[breaklinks=true,letterpaper=true,colorlinks,bookmarks=false]{hyperref}

\theoremstyle{plain}
\newtheorem{theorem}{Theorem}[section]
\newtheorem{lemma}[theorem]{Lemma}
\newtheorem*{lemma*}{Lemma}

\newtheorem*{theorem*}{Theorem}
\newtheorem{corollary}[theorem]{Corollary}
\newtheorem*{corollary*}{Corollary}

\theoremstyle{plain}

\newtheorem*{definition*}{Definition}

\theoremstyle{remark}
\newtheorem{remark}[theorem]{Remark}

\theoremstyle{definition}

\newtheorem{problem}[theorem]{Problem}

\DefineVerbatimEnvironment{code}{Verbatim}{fontsize=\small, frame=single}

\newcommand{\R}{\mathbb{R}}
\newcommand{\N}{\mathbb{N}}
\newcommand{\E}{\mathbb{E}}
\newcommand{\F}{\mathbb{F}}

\newcommand{\ve}{\delta} 
\newcommand{\ur}{\varepsilon} 
\newcommand{\HH}{H_0^1(D)}
\newcommand{\x}{\widehat{x}}
\renewcommand{\u}{\widehat{u}}
\newcommand{\prsp}{(\Omega,\mathcal{U},\mathbb{P})}
\newcommand{\q}{\Tilde{Q}}

\newcommand{\qmlmc}{\widehat{Q}_{L,\{N_l\}}^{\text{ML}}}
\newcommand{\qmpmlmc}{\widehat{Q}_{L,\{N_l\}}^{\text{MPML}}}

\newcommand{\bilf}{\mathcal{A}}
\newcommand{\costdec}{m^{\frac{\gamma-\beta}{2}}}
\newcommand{\costmltot}{C^{\text{ML}}}
\newcommand{\costmpmltot}{C^{\text{MPML}}}

\DeclareMathOperator{\var}{\textrm{var}}

\newcommand{\norm}[1]{\|{#1}\|}
\newcommand{\enorm}[1]{\|{#1}\|_{2}}

\newcommand{\qhel}[1]{\Tilde{Q}_{#1}}
\newcommand{\yhel}[1]{\Tilde{Y}_{#1}}

\newcommand{\rev}[1]{\textcolor{black}{#1}}

\begin{document}

\title{Exploiting Inexact Computations in Multilevel Monte Carlo and Other Sampling Methods
\\[0.5cm] \large{}}

\author{Josef Mart\'inek\thanks{Institute for Mathematics and Interdisciplinary Center for Scientific Computing (IWR), Heidelberg University, 69120, Heidelberg, Germany (martinek@math.uni-heidelberg.de, r.scheichl@uni-heidelberg.de)}, Erin Carson\thanks{Department of Numerical Mathematics, Faculty of Mathematics and Physics, Charles University, Sokolovsk\'a 49/83, 186 75 Praha 8, Czechia (carson@karlin.mff.cuni.cz)}, Robert Scheichl\footnotemark[1]\\
{\tt\small }
}

\maketitle

\input{content.tex}

\printbibliography

\end{document}

%% file: content.tex
\begin{abstract}
Multilevel sampling methods, such as multilevel and multifidelity Monte Carlo, multilevel stochastic collocation, or delayed acceptance Markov chain Monte Carlo, have become standard uncertainty quantification (UQ) tools for a wide class of forward and inverse problems. The underlying idea is to achieve faster convergence by leveraging a hierarchy of models, such as partial differential equation (PDE) or stochastic differential equation (SDE) discretisations with increasing accuracy. By optimally redistributing work among the levels, multilevel methods can achieve significant performance improvement compared to single level methods working with one high-fidelity model. Intuitively, approximate solutions on coarser levels can tolerate large computational error without affecting the overall accuracy. We show how this can be used in high-performance computing applications to obtain a significant performance gain.

As a use case, we analyse the computational error in the standard multilevel Monte Carlo method and formulate an adaptive algorithm which determines a minimum required computational accuracy on each level of discretisation. We show two examples of how the inexactness can be converted into actual gains using an elliptic PDE with lognormal random coefficients. Using a low precision sparse direct solver combined with iterative refinement results in a simulated gain in memory references of up to $3.5\times$ compared to the reference double precision solver; while using a MINRES iterative solver, a practical speedup of up to $1.5\times$ in terms of FLOPs is achieved. These results provide a step in the direction of energy-aware scientific computing, with significant potential for energy savings.
\end{abstract}

\\


\setlength{\emergencystretch}{100pt}
\section{Introduction}
Suppose we are interested in sampling from a probability distribution of a certain quantity $Q$, which depends on the (infinite-dimensional) solution of a partial differential equation  (PDE) or a stochastic differential equation  (SDE). In most cases, direct access to $Q$ is unavailable. Instead, a numerical model is used to obtain a finite-dimensional approximation $Q_L$ of the quantity $Q$. With increasing $L$ the accuracy of the model increases, but so does the cost of computing the solution. This can be a finite element or finite difference method in the case of PDEs, or an Euler-Maruyama discretisation for SDEs. A first idea to approximate the unavailable distribution of $Q$ is to choose a high-fidelity model $Q_L$ with $L$ sufficiently large and to use this model for sampling. Instead, the key idea of multilevel and multifidelity sampling methods, such as multilevel Monte Carlo (MLMC) \cite{giles2015multilevel}, multilevel stochastic collocation \cite{teckentrup2015multilevel}, \rev{multilevel MCMC 
\cite{dodwell2019multilevel}, or multi-index Monte Carlo (MIMC) \cite{haji2016multi}}, is to use a hierarchy of models~$Q_0,\ldots,Q_L$ with increasing accuracy. Then, by combining all the samples from models $Q_0,\ldots,Q_L$ in a suitable way, a significant cost gain can be achieved. \rev{This paper adds one additional layer to the standard analysis, which has, to our knowledge, largely been neglected.} In practice, the finite-dimensional model $Q_l$ is never solved exactly, rather an approximation $\q_l$ is obtained on a computer. We show that multilevel methods often admit safe use of inexact computations on coarse levels $l$ without affecting overall sampling accuracy. \rev{By employing techniques of energy-efficient computing this can lead to significant performance gains.}

\subsection{\rev{Inexact computations and energy-efficient computing}} 
\rev{In the past decades, power has become the principal constraint for computing performance~\cite{horowitz20141}. On the hardware level, this has led to the development of domain specific accelerators which enhance computing performance for specific applications \cite{dally2020domain,jouppi2021tenlessons}. On the software level, attention has recently been brought to energy-aware algorithms for scientific computing. In recent years, there has been growing interest in the selective use of low precision floating point arithmetic to accelerate scientific computations while maintaining acceptable levels of accuracy; cf. \cite{higham2022mixed}.}

\rev{An example of such an approach is mixed-precision iterative refinement, which can be used when the computation of the inexact $\q_l$ involves the solution of a linear system \cite{higham2002accuracy,vieuble2022mixed}. The potential for using iterative refinement in designing energy-efficient linear solvers was studied in \cite{haidar2018design} in the context of $LU$ factorisation of a dense matrix. The FP16-TC dhgesv-TC (Tensor Cores) solver using iterative refinement and half precision for the matrix factorisation achieved more than $5\times$ improvement in terms of energy efficiency than the standard dgesv routine which factorises the matrix in double precision with no iterative refinement.}

\rev{In principle, any iterative procedure can be used to lower the accuracy of the inexact solution $\q_l$. The optimal number of iterations can then be determined via a suitable stopping criterion -- an approach independent of the iterative procedure. To demonstrate the flexibility of such an approach, we consider in this work both direct and iterative solvers for linear systems. To exploit inexact computations efficiently within direct methods, we employ mixed-precision iterative refinement.}

\rev{We discuss in detail how these techniques can be used to improve the performance of multilevel sampling methods. We carry out a detailed analysis for multilevel Monte Carlo and comment on how a similar analysis can be done for multilevel Markov chain Monte Carlo \cite{dodwell2019multilevel} and multi-index Monte Carlo \cite{haji2016multi}.}

\subsection{\rev{Case study: multilevel Monte Carlo}} 

\rev{For a detailed numerical analysis we restrict ourselves to forward uncertainty quantification (UQ) and to multilevel Monte Carlo, tying the required computational accuracy to the discretisation error at level $l$.
We choose a theoretical framework and a suitable error model to quantify the computational error. As a \emph{model problem}, we consider an elliptic PDE with lognormal random coefficients of the form
\begin{equation*}
    -\nabla\cdot\bigl(a(\cdot,\omega)\nabla u(\cdot,\omega)\bigl)=f(\cdot,\omega)
\end{equation*}
depending on the random parameter $\omega$.
%
Such problems arise, for example, in UQ of groundwater flow \cite{cliffe2011multilevel}. The coefficient $a$ and the right-hand side $f$ are assumed to be (infinite-dimensional) random fields. Importantly, the dominant cost lies in the solution of a large system of linear algebraic equations.} 

\rev{Given a (scalar-valued) function $G$ of the solution $u$, such as the  solution at a certain point in the domain, we are interested in sampling from the unavailable distribution of the quantity $Q=G(\omega)$ to compute statistics of this distribution, e.g., the mean $\mathbb{E}[Q]$. In practice, to obtain a computable approximation $Q_L$ of $Q$ we choose the finite element method (FEM), and to approximate the expected value $\mathbb{E}[Q_L]$ we employ a Monte Carlo (MC) method. This introduces a discretisation error and a sampling error. A significant variance reduction can be achieved if the samples are taken on a hierarchy of discretisation levels. This is the underlying idea of the MLMC method~\cite{heinrich2001multilevel,cliffe2011multilevel,giles2015multilevel}.}

\rev{In our model problem, the dominant cost on each level of this hierarchy is the solution of the resulting FE system for each parameter $\omega$. The cost of sampling the input random fields $a(\cdot,\omega)$ and $f(\cdot,\omega)$ can be largely neglected. The error introduced by solving the linear system will be referred to as the computational error. The aim is to balance this computational error with the sampling and discretisation errors, and to employ techniques of energy-efficient computing to obtain performance gains.}

\rev{This is in contrast to applications in computational finance, such as \cite{giles2024rounding}, which is concerned with the analysis of rounding errors in generating random variables in the context of  stochastic differential equations (SDEs) and applications within MLMC. See also the earlier paper \cite{brugger2014mixed}, where the authors explored the use of lower precision on field programmable gate arrays (FPGA) in the MLMC method for SDEs.}

\subsection{\rev{Main contributions and outline of the paper}}

\rev{In this paper, we
\begin{itemize}
    \item establish a theoretical framework for quantifying the computational error in the MLMC method by choosing a suitable error model (Section~\ref{chap_MPMC});
    \item propose a novel adaptive algorithm, which determines the minimum required computational accuracy on each level of discretisation using a-priori error estimates with no additional cost (Section~\ref{sec_MPMLMC_adapt});
    \item provide a theoretical basis for applying this adaptive algorithm to an elliptic PDE with random coefficients and random right-hand sides (Section~\ref{sec_MPMLMCFE});
    \item demonstrate the efficiency of this adaptive algorithm in a sequence of numerical examples, achieving up to a factor of about $3.5 \times$ in simulated memory gain for iterative refinement and a speedup by a factor of about $1.5 \times$ in terms of floating point operations in an iterative solver. A cost analysis and possible use in energy-efficient scientific computing are presented in Section~\ref{sec_cost_analysis}.
\end{itemize}\smallskip}

\rev{The manuscript is divided into five sections. In Section \ref{chap_num_methods}, we discuss linear solvers and various types of errors they incur. We give an overview of floating point arithmetic, the technique of iterative refinement, as well as some convergence results for Krylov subspace methods. Section \ref{sec_fem_rand} introduces the elliptic model problem and its numerical solution via FEM. Selected FE convergence results are presented and the convergence of the FEM with inexact solvers is discussed. After a brief overview of the standard MLMC method, the new theoretical results are presented in Section \ref{sec_inexMLMC}, where we analyse the computational error in MLMC along with the adaptive algorithm and its computational complexity. Finally, the numerical results, as well as a thorough discussion are given in Section \ref{chap_num_results}.}

\section{\rev{Inexact computations in linear solvers}}\label{chap_num_methods}

Let $\widehat{x}$ be a solution of a linear system $Ax=b$ computed by an algorithm. The solution is said to be computed effectively to precision $\ve_e$ if
\begin{equation}\label{eq_eff_sol_linsys}
    \frac{\enorm{b-A\widehat{x}}}{\enorm{b}}\leq C\ve_e
\end{equation}
for a constant $C>0$. Here $\enorm{\cdot}$ denotes the Euclidean norm, although it is possible to use any other norm in principle. The constant $C$ may or may not be dependent on the matrix $A$ and other input data; this is problem-dependent.

\subsection{Floating point arithmetic}\label{sec_fp}

A floating point (FP) number system $\F$ is a finite subset of real numbers whose elements can be written in a specific form; see~\cite{higham2002accuracy} for a thorough description. \rev{Here,} all computations in FP arithmetic are assumed to be carried out under the following standard model: 
For all $x,y\in\F$ 
\begin{equation}\label{eq_standard_model}
    fl(x\,\text{op}\, y)=(1+\nu)(x\,\text{op}\,y),\quad |\nu|\leq \ur,\quad \text{op}=+,-,\times,/,
\end{equation}
where $\ur$ is the unit roundoff. \rev{Since most computations can be decomposed in terms of these basic operations, the above} assumption allows us to analyse the error of a given algorithm. For simplicity, we will often abbreviate ``the floating point arithmetic format with unit roundoff $\ur$'' to ``the precision $\ur$''. 

The standard model is valid in particular for IEEE arithmetic, 
a technical standard of floating point arithmetic, which assumes the preliminaries above and adds other technical assumptions; see \cite{higham2002accuracy} for an overview. The IEEE standard defines several basic formats. In this work, we use formats both standardized and not standardised by IEEE. All of the formats we will use are hardware-supported, \rev{for instance} by the NVIDIA H100 SXM GPU (for specifications see \cite{nvidia_h100_specs}). To be specific, in this manuscript we use quarter (q43), half, single, and double precision. Half, single, and double are basic IEEE formats. The quarter precision format we use has $4$ exponent bits and $3$ significand bits, which means storing one number requires $8$ bits including the sign. The unit roundoff of quarter precision is $2^{-4}$ and its range is $10^{\pm2}$.


Not all arithmetic operations in an algorithm need to be carried out in the same precision. There are techniques which allow us to improve the accuracy of the computed solution via, e.g., iterative refinement, discussed 
in Section \ref{sec_itref}. Moreover, the theoretical error estimates of numerical methods often exaggerate the true error. This motivates us to introduce the so-called effective precision $\ve_e$ in \eqref{eq_eff_sol_linsys}, which 
is not necessarily a hardware or software-supported precision (e.g., like the IEEE standards). It is rather a parameter expressing how accurately the solution is actually computed.

\subsection{Krylov subspace methods}\label{sec_kryl_methods}

Krylov subspace methods are a powerful class of iterative methods 
for large-scale linear systems of equations and eigenvalue problems, particularly those involving sparse or structured matrices. Our application yields matrices which are symmetric positive definite (see Section \ref{sec_fem_rand}), which makes the conjugate gradient (CG) and MINRES methods two suitable Krylov subspace methods for our purpose. The following overview is adapted from \cite[Section 3.1]{greenbaum1997iterative}.

Consider a linear system $Ax=b$ with  symmetric positive definite  $A\in\R^{n\times n}$. Let $x_0\in\R^n$ be an inital estimate of the solution, $e_0=x-x_0$ the initial error, and $r_0 = b-Ax_0$ the initial residual. Then, in the $k$th iteration, CG and MINRES minimize the $A$-norm of the error and the Euclidean norm of the residual over the Krylov subspace $x_0 + \text{span}\{r_0, A^2 r_0, \dots, A^{k-1} r_0\}$, respectively.

\rev{The stopping criterion we use is on the relative residual. It is motivated by our specific use of the linear solver and will allow us to apply abstract error estimates from Section \ref{chap_MPMC}; see also Section \ref{chap_num_results}. To be precise, we require the solution produced by the iterative algorithm to satisfy \eqref{eq_eff_sol_linsys} for a given~$\ve_e>0$, which motivates the choice of MINRES. It has been shown that for Hermitian matrices the residuals produced by CG and MINRES are closely related (see \cite[Exercise 5.1]{greenbaum1997iterative}). Therefore, both methods are expected to perform similarly for our model problem in Section \ref{sec_fem_rand}.}

It is clear that in exact arithmetic, MINRES converges to the true solution in a finite number of iterations. In finite precision arithmetic, this is not the case, and the convergence analysis is nontrivial; see for example \cite[Chapter 4]{greenbaum1997iterative}. We therefore perform all calculations in the MINRES method in double precision. In order to apply the abstract error estimates from Section \ref{chap_MPMC} to MINRES, we assume $\widehat{x}_k\approx x_k$, where $\widehat{x}_k$ and $x_k$ are the MINRES solutions from $k$-th iteration computed in floating point and exact arithmetic, respectively. An alternative approach to performing all calculations in high precision would be to employ a Krylov subspace method coupled with iterative refinement; see Section \ref{sec_itref} for an overview of iterative refinement. Iterative refinement as a technique to improve accuracy of a Krylov subspace method was used, for example, in \cite{carson2017new} and \cite{carson2018accelerating} with the GMRES method.

\subsection{Iterative refinement}\label{sec_itref}

\rev{Consider a linear system $Ax=b$ where $A\in\R^{n\times n}$ and $b\in\R^n$. Iterative refinement is a technique to enhance the accuracy of a numerical solution to this linear system by computing a residual vector $r$ and then correcting the current approximation by solving a linear system with right hand side $r$ to reduce the error; see~\cite[Section 12]{higham2002accuracy} for a detailed overview. In this way, the process can be repeated iteratively until a limiting level of accuracy is achieved.
The particular version of iterative refinement used in this work is presented in Algorithm \ref{alg_itref}. It is a special case of \cite[Algorithm 1.1]{carson2018accelerating}.}  
\begin{algorithm}[t]
\caption{Iterative refinement}
\label{alg_itref}
\begin{algorithmic}[1]
    \REQUIRE $A\in\mathbb{R}^{n\times n}$, $b\in\mathbb{R}^n$, both in precision $\ur$, tolerance $\ve_e>0$.
    \ENSURE an approximation $x_i$ of the solution $x$ stored in precision $\ur$.
    \STATE Factorise $A$ in precision $\ur_f$. \label{aux26}
    \STATE Solve $Ax_0=b$ in precision $\ur_f$ by substitution and store $x_0$ in precision $\ur$.
    \FOR{$i=0$ to $i_{\max}$}
        \STATE Compute $r_i=b-Ax_i$ in precision $\ur_r$ and round $r_i$ to precision $\ur_s$. \label{aux_step4}
        \IF{$\norm{r_i}/\norm{b} < \ve_e$} \label{aux28}
            \STATE Exit algorithm. \label{aux29}
        \ENDIF
        \STATE Solve $Ad_i=r_i$ by (forward/backward) substitution in precision $\ur_f$ or $\ur$\\ (using the factorisation computed in step \ref{aux26}) and store $d_i$ in precision $\ur$. \label{aux27}
        \STATE $x_{i+1}=x_i+d_i$ in precision $\ur$.
    \ENDFOR
\end{algorithmic}
\end{algorithm}

\rev{Algorithm \ref{alg_itref} uses the following three precisions:} 
\begin{itemize}
    \item $\ur$ is the (working) precision \rev{that the data $A$, $b$ and solution $x$ are stored in}, 
    \item $\ur_f$ is the precision \rev{that the factorisation of $A$ is computed in},
    \item $\ur_r$ is the precision that residuals are computed in.
    \item \rev{$\ur_s$ is the precision that the correction equation is solved in}.\footnote{Note that in this subsection and in Section \ref{sec_fem_rand} the symbol $\ur_s$ does not denote single precision here. We use it to keep the indices in accordance with \cite{carson2018accelerating}, where the precisions are denoted by $u_r$, $u$, $u_f$, and $u_s$, respectively.}
\end{itemize}

\rev{The precisions $\ur_r$, $\ur$, $\ur_s$, and $\ur_f$ take only the values quarter, half, single, or double precision in this work, and we assume that $\ur_r\leq\ur\leq\ur_s\leq\ur_f$. Since we assume a direct linear solver with Cholesky factorisation, we effectively can choose  $\ur_s=\ur_f$ in our case. From the perspective of the limiting tolerance for the residual norm $\norm{r_i}/\norm{b}$, no numerical benefit is obtained by choosing $\ur_s<\ur_f$. Nevertheless, it turns out that using a higher precision in step \ref{aux27} can help in practice to reach the same desired tolerance in fewer iterations with negligible additional cost per iteration.} There is also no reason to use extra precision to compute the residual in step \ref{aux_step4} of the algorithm, since we are only interested in bounding the backward error. Computing the residual in higher precision is useful if bounding the forward error is of interest; see \cite{carson2018accelerating}.

\rev{Algorithm \ref{alg_itref} is used extensively in this work to achieve the desired accuracy of the computed solution. Additional benefits arise in terms of cost. The key point is to reuse the factorisation computed in step \ref{aux26} of the algorithm in step \ref{aux27} of the algorithm. By reusing the factorisation, the accuracy of the solution can be improved with little additional cost, since} the factorisation is typically the dominant part in terms of the required number of operations; see \cite[Section 2.2]{vieuble2022mixed}.

\rev{As in our Krylov subspace method, the stopping criterion in Algorithm \ref{alg_itref} is on} the relative residual norm. It is again motivated by our specific use of iterative refinement; see Sections \ref{chap_MPMC} and \ref{chap_num_results}. From the definition of the stopping criterion in step \ref{aux28} it follows that if Algorithm \ref{alg_itref} converges, then the produced solution is computed effectively to precision $\ve_e$ in the sense of \eqref{eq_eff_sol_linsys}. This will allow us to apply abstract error estimates from Section \ref{chap_MPMC} to iterative refinement. The convergence of Algorithm \ref{alg_itref} is analysed thoroughly in \cite{carson2018accelerating}. \rev{We will return to it 
in Section \ref{sec_fem_rand_fin_prec}.}

\section{Finite element methods for PDEs with random data}\label{sec_fem_rand}
As a use case 
for inexact computations within multilevel sampling methods, we now consider the multilevel Monte Carlo method and a standard elliptic PDE with random data, discretised via finite elements (FE). Problems of this kind arise, for example, in geosciences, namely, in the study of groundwater flow; see \cite{scheidegger1957physics, hoeksema1985analysis, freeze1975stochastic}, and the references therein. 

\subsection{An elliptic PDE with random data and its FE solution}
We now state the standard weak formulation of an elliptic PDE with random data: 

\begin{problem}[AVP with random data]\label{pr_AVP_rnd}
Let $\prsp$ be a probability space and 
$V\coloneqq\HH$ where $D$ is a bounded domain in $\mathbb{R}^d$, $d=2,3$. 
For $\omega\in\Omega$, we define $\bilf:V\times V\times\Omega \rightarrow\R$ and $l:V\times\Omega \rightarrow \R$ by
\begin{align}\label{eq_avp_rnd}
    \begin{split}
        \bilf\bigl(u(\cdot,\omega),v,\omega\bigr)&\coloneqq\int_D a(x,\omega)\nabla u(x,\omega)\cdot\nabla v(x)\text{d}x\quad\text{and}\\ l(v,\omega)&\coloneqq\int_Df(x,\omega)v(x)\text{d}x,
    \end{split}
\end{align}
respectively, where $a$ and $f$ are fixed random fields satisfying $a(\cdot,\omega)\in L^\infty(D)$ and $f(\cdot,\omega)\in L^2(D)$. A function $u(\cdot,\omega)\in V$ is a solution of this \emph{abstract variational problem} if it satisfies for a.e.\ $\omega\in\Omega$
\[
\bilf\bigl(u(\cdot,\omega),v,\omega\bigr)=l(v,\omega) \quad \text{for all} \ \ v\in V. 
\]
\end{problem}

In order to prove the unique solvability of the AVP with random data, we assume that there exist $a_{min}$ and $a_{max}$ such that  
\begin{equation}\label{eq_ass_coeff_bound_rnd}
    0<a_{min}\leq a(x,\omega)\leq a_{max}<\infty \quad \text{for a.e.} \ \  \omega\in\Omega \ \ \text{and a.e.} \ \ x\in D.
\end{equation}
Under the assumption \eqref{eq_ass_coeff_bound_rnd} the unique solvability of the AVP with random data (Problem~\ref{pr_AVP_rnd}) can be proved sample-wise using the Lax-Milgram lemma in the standard way; see \cite{babuska2004galerkin}. The coercivity and continuity constants in the Lax-Milgram lemma do not depend on $\omega$. More details regarding the analysis of Problem \ref{pr_AVP_rnd} in the stochastic context can be found in \cite{babuska2004galerkin, barth2011multi}. Let us note that if the bounds in \eqref{eq_ass_coeff_bound_rnd} are weakened and depend on $\omega$, the analysis is still possible but it becomes more complicated; see~\cite{teckentrup2013further} for a thorough discussion.

To obtain a FE error estimate in the $L^2$ norm, the domain $D$ is assumed to be convex and the random field $a$ is assumed to be uniformly Lipschitz continuous, i.e., 
there exists $L>0$ such that 
\begin{equation}\label{eq_ass_higher_regularity_rnd}
    |a(x_1,\omega)-a(x_2,\omega)|\leq L\|x_1-x_2\| \quad \text{for a.e.} \ \  \omega\in\Omega \ \ \text{and a.e.} \ \ x_1,x_2 \in D.
\end{equation}

In order to solve Problem \ref{pr_AVP_rnd}, we use conforming FEs on a shape-regular and quasi-uniform family of triangulations of the domain $D\subset\R^n$ (in our examples $n=2$) with a mesh parameter $h>0$. We employ (piecewise linear) $\mathcal{P}_1$ Lagrange elements to compute the \emph{discrete solution} $u_h(\cdot,\omega)$ to the AVP with random data (Problem~\ref{pr_AVP_rnd}) in the FE space $V_h \subset V$. Using a basis of $V_h$ consisting of hat functions $\phi_j$, this is equivalent to solving a linear system $A(\omega)x(\omega)=b(\omega)$ with a positive definite matrix~$A(\omega)$;
see \cite{ern2004theory} for details.

\subsection{Approximate FE solutions using inexact solvers}\label{sec_fem_rand_fin_prec}

 Due to the limitations of inexact linear solvers, the discrete solution $u_h$ is not obtained exactly in practice. Instead, we compute an approximation $\u_h$. The aim of this section is to estimate the error $\|u_h(\cdot,\omega)-\u_h(\cdot,\omega)\|_{\HH}$ by means of the residual of the solution of the FE system $Ax=b$ and investigate how iterative refinement from Section \ref{sec_itref} can be employed to solve this system. Let us first introduce some notation.

Let $\u_h$ be the approximation of the discrete solution $u_h$ to Problem \ref{pr_AVP_rnd} such that%
\begin{equation}\label{eq_inexact_fe_sol}
    \u_h(\cdot,\omega)=\sum_{j=1}^n \x_j(\omega)\phi_j
\end{equation}
and let $r(\omega)\coloneqq A(\omega)\x(\omega)-b(\omega)$ denote the residual.

\begin{lemma}\label{th_fem_resiudal_est_rand}
    Let $u_h$ be the discrete solution of 
    Problem \ref{pr_AVP_rnd} and let $\u_h$ be the approximation of $u_h$ defined above. Then 
    \begin{equation*}
        \|u_h(\cdot,\omega)-\u_h(\cdot,\omega)\|_{\HH}\leq C\|f(\cdot,\omega)\|_{L^2(D)}\frac{\enorm{r(\omega)}}{\enorm{b(\omega)}} \quad \text{for a.e.} \ \  \omega\in\Omega,
    \end{equation*}
    where $C$ is independent of $h$, $u$, and $\omega$ and $\enorm{\cdot}$ denotes the Euclidean norm on $\R^n$.
\end{lemma}
\begin{proof}
    For $\omega\in\Omega$ fixed,
    the bound follows from \cite[Proposition 9.19]{ern2004theory}, with a constant independent of $u$ and $h$. 
To see that $C$ is also independent of $\omega$, we write out the constant $C$ from \cite[Proposition 9.19]{ern2004theory} explicitly: 
\begin{equation}\label{eq_FE_in_FP_constant}
        C=\frac{\kappa(\mathcal{M}_t)^{1/2}}{c_{tP}\alpha}.
    \end{equation}
    Here, $c_{P}$ is the Poincar\'e constant, 
    $\alpha$ is the coercivity constant of the bilinear form~$\bilf$~and~$\kappa(\mathcal{M}_t)$ is the condition number of the mass matrix $\mathcal{M}_t$. It follows from the definition of $\mathcal{M}_t$ that $\kappa(\mathcal{M}_t)$ is independent of $\omega$; see 
    \cite[Section 9.1.3]{ern2004theory}. Due to~\eqref{eq_ass_coeff_bound_rnd}, 
    $\alpha$ is also independent of $\omega$. This completes the proof.
\end{proof}

Thus, 
to control the error $\|u_h(\cdot,\omega)-\u_h(\cdot,\omega)\|_{\HH}$ it suffices to control the relative residual $\enorm{r(\omega)}/\enorm{b(\omega)}$ of the resulting linear system. In the following we derive the convergence rate of iterative refinement from Section \ref{sec_itref} applied to the~FE~system from Problem \ref{pr_AVP_rnd}. This will guide our choice of the precisions in iterative refinement (Algorithm \ref{alg_itref}). To proceed with our analysis, we need some auxiliary inequalities from \cite{ern2004theory}, which we summarize here.
\begin{lemma}\label{th_aux_fe_inequalities}
    Let $A(\omega)x(\omega)=b(\omega)$ be the FE system corresponding to the approximate FE solution $\u_h$ of Problem \ref{pr_AVP_rnd} as in \eqref{eq_inexact_fe_sol}. The following estimates hold:
    \begin{enumerate}
        \item $\norm{A(\omega)}_2\leq c$, and \label{aa1}
        \item $\kappa_2\big(A(\omega)\big)\leq ch^{-2}$, \label{aa2}
    \end{enumerate}
    where $\kappa_2$ is the condition number of $A$ with respect to the spectral norm and the generic constant $c$ is independent of the discretisation parameter $h$ and of $\omega$.
\end{lemma}
\begin{proof}
The first claim follows from 
the proof of \mbox{\cite[Theorem 9.11]{ern2004theory}} (the last inequality in the first part) together with \cite[Theorem 9.8]{ern2004theory}. The inequality from \mbox{\cite[Theorem 9.11]{ern2004theory}} gives us a bound on $\norm{A}_2$ using $h$ and the eigenvalues of the mass matrix, which then can be bounded using \cite[Theorem 9.8]{ern2004theory}. 

The second claim follows from \cite[Theorem 9.11]{ern2004theory} (with $s=t=1$) as well as \mbox{\cite[Example 9.13]{ern2004theory}}. The $\omega$-independence of the constant $c$ follows as in Lemma \ref{th_fem_resiudal_est_rand} from the definition of the stiffness matrix $A$ and from \eqref{eq_ass_coeff_bound_rnd}.
\end{proof}

As a consequence of Lemma \ref{th_aux_fe_inequalities} we formulate a corollary about the convergence of iterative refinement for the FE system. It follows immediately from \cite[Corollary 4.2]{carson2018accelerating} using the inequalities from Lemma \ref{th_aux_fe_inequalities} and the fact that $\|x(\omega)\|_2\leq\|A^{-1}(\omega)\|_2\|b(\omega)\|_2$. For this corollary we make the reasonable assumption that \rev{the approximate solution $x_i(\omega)$ after $i$ steps of iterative refinement satisfies $\enorm{x_i(\omega)}\approx\enorm{x(\omega)}$. The analysis in \cite{carson2018accelerating} is in the maximum norm $\|\cdot\|_\infty$, but it is easily adapted to $\|\cdot\|_2$.}

\begin{corollary}\label{th_itref_fe_convergence}
\rev{Let $A(\omega)x(\omega)=b(\omega)$ be as in Lemma \ref{th_aux_fe_inequalities} and let $c_1$ and $c_2$ be as in \cite[eq. (2.4)]{carson2018accelerating}.
 If $(c_1\kappa_2(A(\omega))+c_2)\ur_s\ll 1$, then there exists a constant $c>0$ such that the residual in iterative refinement 
 (Algorithm \ref{alg_itref}) satisfies
 \begin{equation}\label{aux_itref1}
    \|r_{i}(\omega)\|_2 \le c (1+h^{-2}) \max\Big( \ur_s\|r_{i-1}(\omega)\|_2, \ur \enorm{b(\omega)} \Big).
    \end{equation}
    where $c$ only depends on generic constants in \cite[Corollary 4.2]{carson2018accelerating} and Lemma \ref{th_aux_fe_inequalities}.}
\end{corollary}
\rev{This corollary allows us to give a priori upper bounds on the precisions $\ur_r$, $\ur$, and $\ur_f$ in Algorithm \ref{alg_itref}  that guarantee that the stopping criterion $\enorm{r_i(\omega)}/\enorm{b(\omega)}<\ve_e$ in Line 6 of Algorithm \ref{alg_itref} is satisfied. In particular, it suffices that $\ur_r$, $\ur$, and $\ur_f$ satisfy the following three conditions:
\begin{enumerate}
    \item $(c_1\kappa_2(A(\omega))+c_2)\ur_s\ll 1$,
    \item $c (1+h^{-2})\ur <\ve_e$,
    \item $\ur_r\leq\ur$.
\end{enumerate}
Under these conditions, Algorithm \ref{alg_itref} converges to the required tolerance $\ve_e$ and \eqref{aux_itref1} provides a bound on the convergence rate.}

\section{Inexact computations in multilevel Monte Carlo}
\label{sec_inexMLMC}
\subsection{Standard multilevel Monte Carlo}\label{sec_MLMC}

Multilevel Monte Carlo (MLMC) is one of the standard multilevel sampling methods to compute expectations for UQ in PDE applications, and a suitable example to show how inexact computations can be leveraged in that context. The main idea of MLMC is to optimally balance the sampling and discretisation errors of a hierarchy of approximate models; see \cite{giles2015multilevel} for an overview. We formulate it here in full generality:

Suppose $Q:\Omega\rightarrow\R$ is a random variable and we are interested in computing its mean $\E[Q]$. Assume that $Q$ cannot be evaluated sample-wise, but a sequence of  
models $Q_l, l\in\{0,\ldots,N\}$, approximating $Q$ with increasing accuracy, is available. We define the following auxiliary Monte Carlo (MC) estimators:
\begin{equation*}
    \widehat{Y}_0\coloneqq\frac{1}{N_0}\sum_{k=1}^{N_l} Q_{0}^{(k)} \quad \text{and} \quad \widehat{Y}_l\coloneqq\frac{1}{N_l}\sum_{k=1}^{N_l} \bigl(Q_{l}^{(k)}-Q_{{l-1}}^{(k)}\bigr),
\end{equation*}
where $l=1\ldots,L$. The estimator 
\begin{equation}\label{eq_mlmcest}
    \qmlmc\coloneqq\sum_{l=0}^L \widehat{Y}_l
\end{equation}
will then be referred to as the MLMC estimator for $E[Q]$. Although not necessary, the estimators $\widehat Y_l$ are, in this work, assumed to be independent. In this definition, it is assumed that the models can be evaluated \emph{exactly}, i.e., we can compute $Q_l(\omega)$. In the following we sometimes abuse notation and denote by $\qmlmc$ an estimate computed using the estimator $\qmlmc$.

In practice, we use an adaptive algorithm to determine the values $L$ and $\{N_l\}_{l=0}^L$ for a given tolerance $\text{TOL}$; see \cite[Section 2]{cliffe2011multilevel}). The adaptive MLMC algorithm aims to compute the optimal values of $L$ and $\{N_l\}_{l=0}^L$ by minimizing the computational cost for a given variance. 
To this end, the algorithm uses sample averages $\widehat Y_l$ (see~\eqref{eq_mlmcest}) and variance estimators
\begin{equation}\label{eq_sample_var}
    s_l^2\coloneqq\frac{1}{N_l}\sum_{k=1}^{N_l}\bigl(Y_l^{(k)}-\widehat Y_{l}\bigr)^2
\end{equation}
of the random variables $\{Y_l\}_{l=0}^L$. What we also need in the algorithm is the cost $C_l$ to evaluate $Q_l(\omega)$ for each sample $\omega\in \Omega$. 

The MLMC complexity can be analysed by imposing standard assumptions on $Q$~and~$Q_l$. Namely, we assume that there exists $m>1$ and $\alpha, \beta, \gamma>0$ such that
\begin{equation}\label{eq_mlmc_biasvarcost}
    |\E[Q_{l}-Q]|=O(m^{-\alpha l}),\quad
    \var[Y_l]=O(m^{-\beta l}), \quad \text{and} \quad
    C_l=O(m^{\gamma l}).
\end{equation}
It can be shown that three complexity regimes can be distinguished based on the values of $\beta$ and $\gamma$. The optimal performance of the MLMC algorithm is achieved when $\beta>\gamma$, i.e., when variance decays faster than cost increases. The full complexity analysis can be found in \cite[Theorem 1]{cliffe2011multilevel}.

\subsection{Computational error in multilevel Monte Carlo}\label{chap_MPMC}

In this section we present a convergence analysis taking into account the computational error in the MLMC method discussed in Section \ref{sec_MLMC}. Our analysis shows that especially on the coarser levels, it is sufficient to compute with relatively low effective precision. 
Significant gains both in terms of memory and computational time can be obtained, depending on the model considered and its implementation. To this end, we propose a novel adaptive MLMC algorithm, which determines the minimum required numerical precision with no additional cost. This algorithm will be referred to as \emph{mixed-precision multilevel Monte Carlo (MPML)} for simplicity. The efficiency of the adaptive algorithm is demonstrated on numerous experiments in Section \ref{chap_num_results}.

Let us start by defining the setting. As in Section \ref{sec_MLMC}, we assume $Q:\Omega\rightarrow\R$ to be a random variable and the goal is to compute its mean $\E[Q]$. This time we assume that we cannot evaluate $Q_l(\omega)$ exactly, but only an approximation $\Tilde{Q}_l$ of $Q_l$. This might be due to finite precision arithmetic for the model evaluation or for generating the random variable. Using $\Tilde{Q_l}$ instead of $Q_l$ in the MLMC estimator \eqref{eq_mlmcest} gives us what will be referred to as the mixed-precision multilevel Monte Carlo estimator and will be denoted by $\qmpmlmc$. The crucial difference between the MPML estimator and the standard MLMC estimator is the error model which is used. For MPML we propose an additive error model stated below. Using the bias-variance decomposition we can also quantify the overall error of the MPML estimator.

\begin{theorem}[Computational error in MPML]\label{th_abstractMPMLMCcomplex3}
Let $m\in\N$, $m>1$ and assume that  $\ve_0,\ve_1,\ldots$ is a sequence of parameters characterising computational error with $1>\ve_l>0$. Assume that there exist $\alpha_1,\beta_1,\alpha_2,\beta_2>0$ such that
\begin{align}
    |\E[\Tilde{Q}_{l}-Q]|&=O(m^{-\alpha_1l}+\ve_l^{\alpha_2}),\label{eq_mpmlmcbiasdec3}\\
    \var[\Tilde{Y}_{l}]&=O(m^{-\beta_1l}+\ve_l^{\beta_2}).\label{eq_mpmlmcvardec3}
\end{align}
Let $L\in\N$ and $N_0,\ldots,N_L\in\N$ and let $\qmpmlmc$ be the corresponding MPML estimator. Then, the MSE of this estimator satisfies
\begin{equation*}
    \E\bigl[(\E[Q]-\qmpmlmc)^2\bigr]\leq C\biggl(\Bigl(m^{-2\alpha_1 L}+\sum_{l=0}^L\frac{m^{-\beta_1 l}}{N_l}\Bigr)+\Bigl(m^{-\alpha_1 L}\ve_L^{\alpha_2}+\ve_L^{2\alpha_2}+\sum_{l=0}^L\frac{\ve_l^{\beta_2}}{N_l}\Bigr)\biggr).
\end{equation*}
\end{theorem}
\begin{proof}
    The inequality follows directly from the bias-variance decomposition 
    \begin{equation*}
    \rev{\E\bigl[(\E[Q]-\qmpmlmc)^2\bigr]=(\E[Q-\Tilde{Q}_{L}])^2+ \sum_{l=0}^L\frac{\var[\Tilde{Y}_l]}{N_l}}
\   \end{equation*}
    (see also \cite{cliffe2011multilevel}) and the assumptions \eqref{eq_mpmlmcbiasdec3} and \eqref{eq_mpmlmcvardec3}.
\end{proof}

This general error estimate can be used in practice to compute a bound on the computational error so that the overall MSE does not exceed a certain tolerance. A~concrete example of what form the parameter $\ve_l$ can take will be discussed in Section \ref{sec_MPMLMCFE}. For the purpose of this general error estimate we have not assumed that the computational error decays. An example of this would be the situation when we have a hierarchy of approximations of a linear PDE and solve the resulting system of linear algebraic equations on each level using an iterative solver with a fixed number of iterations, or a sparse direct solver in finite precision. Then the error of the model decays with increasing $l$, but we expect the computational error to grow. This is in accordance with what has been observed in \cite[Figure 3.1]{giles2024rounding}.

In our PDE problem, we assume a hierarchy of discrete FE models based on uniform mesh refinement with a factor $m>1$ of an inital mesh with mesh size $h_0$, such that $h_l := h_0 m^{-l}$ for $l \in \mathbb{N}$. Assumptions \ref{eq_mpmlmcbiasdec3} and \ref{eq_mpmlmcvardec3} may then be rewritten as
\begin{equation}\label{eq_mpmlbiasvardec_fem_general}
    |\E[\Tilde{Q}_{l}-Q]|=O(h_l^{\alpha_1}+\ve_l^{\alpha_2}), \quad \text{and} \quad
    \var[\Tilde{Y}_{l}]=O(h_l^{\beta_1}+\ve_l^{\beta_2}).
\end{equation}
\rev{In this case Theorem \ref{th_abstractMPMLMCcomplex3} can be stated in the following form.
\begin{corollary}\label{cor:abstractMPMLMCcomplex_fem}
    Under Assumption \eqref{eq_mpmlbiasvardec_fem_general}, the MSE of 
    $\qmpmlmc$ satisfies
    \begin{equation*}
        \E\bigl[(\E[Q]-\qmpmlmc)^2\bigr]\leq C\biggl(h_L^{2\alpha_1}+\sum_{l=0}^L\frac{h_l^{\beta_1}}{N_l}\Bigr)+\Bigl(h_L^{\alpha_1}\ve_L^{\alpha_2}+\ve_L^{2\alpha_2}+\sum_{l=0}^L\frac{\ve_l^{\beta_2}}{N_l}\Bigr)\biggr).
    \end{equation*}
\end{corollary}}

\begin{remark}[\rev{Asymptotic MPML cost}]
    Note that the asymptotic cost of the MPML method remains the same as for the standard MLMC method as long as $\ve_L^{\alpha_2}=O(h_L^{\alpha_1})$ and $\ve_l^{\beta_2}=O(h_l^{\beta_1})$. However, due to the use of lower precision arithmetic, the overall computational time can be reduced significantly which is the topic of the next section.
\end{remark}

\subsection{Adaptive MPML algorithm}\label{sec_MPMLMC_adapt}

In this section we develop an adaptive algorithm which will automatically choose a suitable \rev{computational accuracy} on each level of the MLMC method. We will use the standard MLMC algorithm as the foundation for our proposed algorithm.

In order to choose the correct \rev{computational accuracy} in each step, we will use the error bound for the MPML method from \rev{Corollary \ref{cor:abstractMPMLMCcomplex_fem}}. We propose the following approach: choose the accuracy $\ve_l$ on level $l$ such that the total MSE of the MPML estimator is not greater than a constant times the MSE of the standard MLMC estimator for a fixed \mbox{constant $k_p\in(0,1)$}, the choice of which we discuss later. According to \rev{Corollary \ref{cor:abstractMPMLMCcomplex_fem}}, for this to hold it suffices to choose $\ve_l,\; l=0,\ldots,L$ such that
\begin{equation*}
    h_L^{\alpha_1}\ve_L^{\alpha_2}+\ve_L^{2\alpha_2}+\sum_{l=0}^L\frac{\ve_l^{\beta_2}}{N_l}\leq k_p\Bigl(h_L^{2\alpha_1}+\sum_{l=0}^L\frac{h_l^{\beta_1}}{N_l}\Bigr)
\end{equation*}
for a fixed constant $k_p\in(0,1)$. To balance the terms in the error estimate, it is sufficient to choose $\ve_l,\;l=0,\ldots,L$ such that
\begin{equation}
    h_L^{\alpha_1}\ve_L^{\alpha_2}+\ve_L^{2\alpha_2}  \leq k_p h_L^{2\alpha_1} \quad \text{and} \quad
\label{eq_coarse_prec_ineq}
    \sum_{l=0}^L\frac{\ve_l^{\beta_2}}{N_l} \leq k_p\sum_{l=0}^L\frac{h_l^{\beta_1}}{N_l}.
\end{equation}
Since for $\ve_L^{\alpha_2}< h_L^{\alpha_1}$, we have $\ve_L^{2\alpha_2}\ll h_L^{\alpha_1}\ve_L^{\alpha_2}$, it suffices to choose
\begin{equation*}
    \ve_L\leq\Big(k_ph_L^{\alpha_1}\Big)^{1/\alpha_2}.
\end{equation*}
Moreover, in order to satisfy \eqref{eq_coarse_prec_ineq} we can choose $\delta_l$ as
\begin{equation}\label{eq_prec_choice}
    \begin{split}
        \ve_l&\coloneqq \Big(k_p h_l^{\beta_1}\Big)^{1/\beta_2},\quad l=0,\ldots,L-1,\\
        \ve_L&\coloneqq \min\biggl\{\Big(k_p h_L^{\beta_1}\Big)^{1/\beta_2},\Big(k_p h_L^{\alpha_1}\Big)^{1/\alpha_2}\biggr\}.
    \end{split}
\end{equation}
With this choice, both inequalities in 
\eqref{eq_coarse_prec_ineq} are satisfied and we obtain the desired error estimate \rev{from Corollary \ref{cor:abstractMPMLMCcomplex_fem}}. 
\begin{remark}[\rev{Balancing the computational and model error}]
    The reason we ``hide" the computational error rather than optimally balancing it with the model error in our adaptive algorithm is that in the case when the computational error comes from a linear solver, it typically decays exponentially with the number of iterations (e.g., in iterative refinement; see Section \ref{sec_itref}) and therefore balancing the two errors would not bring great benefits. However, it might be of interest in cases when the computational error decreases polynomially.
\end{remark}

\begin{algorithm}[t]
\caption{Adaptive MPML algorithm}
\label{alg_MPMLMC}
\begin{algorithmic}[1]
    \REQUIRE $m$, $\text{TOL}$, $L=1$, $L_{\text{max}}$, $N_0=N_1=N_{\text{init}}$
    \ENSURE $\qmpmlmc,\{N_l\}$
    \WHILE{$L \leq L_{\text{max}}$}
        \STATE Compute $\ve_l,\; l=0,\ldots,L$, using \eqref{eq_prec_choice}\label{alg_step_prec_choice}
        \FOR{$l=0$ to $L$}
            \STATE Compute $N_l$ new samples $Y_l^{(k)}$ using \eqref{eq_mlmcest} with computational accuracy $\ve_l$\label{alg_step_samples_comp}
            \STATE Compute $\widehat Y_l$, $s^2_l$ and estimate $C_l$
        \ENDFOR
        \STATE Update estimates for $N_l$ as $N_l \coloneqq \sqrt{\frac{V_l}{C_l}} \frac{2}{\text{TOL}^2} \sum_{k=0}^{L} \sqrt{V_k C_k}$\label{alg_step_Nl_update}
        \IF{$|\widehat Y_L| > \frac{rm^\alpha-1}{\sqrt{2}} \text{TOL}$}
            \STATE $L \coloneqq L+1$
            \STATE $N_L \coloneqq N_{\text{init}}$
        \ENDIF
        \IF{$|\widehat Y_L| \leq \frac{rm^\alpha-1}{\sqrt{2}} \text{TOL}$ \AND $\sum_{l=0}^{L} s^2_l / N_l \leq \text{TOL}^2/2$}
            \STATE $\qmpmlmc \coloneqq \sum_{l=0}^{L} \widehat Y_l$
        \ENDIF
    \ENDWHILE
\end{algorithmic}
\end{algorithm}

Let us discuss in more detail the choice of the constant $k_p$. Although in this work the constant $k_p$ is chosen to be fixed, more general choices are possible in principle. The value $k_p\coloneqq 0.05$ is a safe choice to bound the computational error, as demonstrated in Section \ref{chap_num_results}. Note also that the values of $\delta_l$ can be computed ``on the fly'' with no additional cost, given that the decay rates of bias and variance in \rev{\eqref{eq_mpmlbiasvardec_fem_general}} are known. \rev{For the resulting adaptive MPML algorithm see Algorithm \ref{alg_MPMLMC}.}

It is natural to ask how the choice of the constant $k_p$ affects the number of samples~$N_l$ required on each level $l$ to achieve the desired tolerance. According to Corollary \ref{cor:abstractMPMLMCcomplex_fem}, if the accuracy $\ve_l$ is chosen according to \eqref{eq_prec_choice} then the variance is increased on each level at most by approximately the factor $(1+k_p)$. This means that the number of samples $N_l$ on each level is increased at most by the same factor~$(1+k_p)$; see step \ref{alg_step_Nl_update} of Algorithm \ref{alg_MPMLMC}.
Since $k_p\ll 1$, this does not pose a problem for us. Depending on the exact settings of the problem (on the linear solver), it might make sense to use a different cost model for MPML than for MLMC to estimate the cost per sample $C_l$ on each level, which may impact the number of samples on each level as well. However, in Figure \ref{fig_mpml_minres_nsamp} we verify numerically that in our example the overall increase in the number of samples compared to standard MLMC is negligible.

\subsection{Cost analysis}\label{sec_cost_analysis}
The cost gain using the adaptive MPML algorithm from Section \ref{sec_MPMLMC_adapt} depends on which of the three complexity regimes of the MLMC estimator 
 in \cite[Theorem 1]{cliffe2011multilevel} applies. These depend on the relative sizes of $\beta$ and $\gamma$ in \eqref{eq_mlmc_biasvarcost}. 
 
 Intuitively, one obtains the most significant gains in cases where the cost on the coarser levels dominates. This is due to the fact that on the coarser levels the discretisation error is larger and therefore the estimator can also tolerate a larger computational error without affecting the overall accuracy; see \eqref{eq_prec_choice}. The cost on the coarser levels dominates in the case when $\beta>\gamma$ in \eqref{eq_mlmc_biasvarcost}, i.e., when variance decays faster than the cost increases (see below for a more precise statement).

The abstract cost analysis is done here in terms of arbitrary cost units. In applications, we may consider, e.g., CPU time, memory references, or floating point operations, depending on what best fits our purpose. At the end of this subsection we apply the abstract cost analysis to the Problem \ref{pr_qoi} and give a concrete example of cost measures for this problem.

For the purpose of the abstract analysis we assume the following. For the standard MLMC we assume $\var[Y_l]=c_v m^{-\beta l}$ and $C_l=c_c m^{\gamma l}$ with $\beta>\gamma$; see \eqref{eq_mlmc_biasvarcost} and \cite[Theorem 1]{cliffe2011multilevel}. Further, we assume that both MLMC and MPML algorithms use the same number of samples $N_l$ on each level and the variance on each level $\var[Y_l]$ is the same for both algorithms. This is a reasonable assumption, since our adaptive MPML algorithm chooses the computational error significantly smaller than the model error (see the discussion in Section \ref{sec_MPMLMC_adapt} and Figure \ref{fig_mpml_minres_nsamp}). For the costs per sample on each level with low accuracy
we assume only that
\begin{equation}\label{eq_cost_saving_ass}
       C^{\text{MP}}_0 \le q C_0, \quad
       C^{\text{MP}}_l \leq C_l,\quad l\geq1,
\end{equation}
where $q\in(0,1)$ is the factor by which the coarsest level cost is reduced. 

The total cost per level in MLMC decays as
\begin{equation}\label{eq_ml_level_cost_red}
    \frac{C_{l+1}N_{l+1}}{C_l N_{l}}=\costdec,
\end{equation}
where we used the definition of $N_l$ from step \ref{alg_step_Nl_update} of Algorithm \ref{alg_MPMLMC} and the bias and variance decay rates. 
We see that indeed the coarsest level cost dominates in the regime $\beta>\gamma$. For the total cost of the MLMC algorithm we get
\begin{equation*}
    \costmltot=\sum_{l=0}^L C_l N_l
    = 
    C_0N_0\frac{1-\bigl(\costdec\bigr)^L}{1-\costdec}.
\end{equation*}
In summary, using \eqref{eq_cost_saving_ass}, the ratio of the total costs of the two estimators is therefore bounded by
\begin{equation*}
    \frac{\costmpmltot}{\costmltot}\leq q\frac{1-\costdec}{1-\bigl(\costdec\bigr)^L}+\costdec\frac{1-\bigl(\costdec\bigr)^{L-1}}{1-\bigl(\costdec\bigr)^L}.
\end{equation*}
Letting $L\rightarrow\infty$ we obtain an asymptotic bound $\frac{\costmpmltot}{\costmltot}\leq q + \costdec (1-q)$. For $\beta>\gamma$ this bound is less than $1$. We summarize this in the following corollary.
\begin{corollary}\label{cor_mpml_mlmc_cost_savings}
    Assume that by using MPML (Algorithm \ref{alg_MPMLMC}), the computational cost on the coarsest level is reduced by a factor $q\in(0,1)$ and that doing so does not (significantly) change the number of samples or the variance on any level compared to standard MLMC. If the variance decays sufficiently fast, i.e., $\beta>\gamma$, then the total cost of MPML is reduced asymptotically as $L\rightarrow\infty$ by at least a factor $q + \costdec (1-q)$ compared to standard MLMC. 
\end{corollary}

By standard MLMC we mean Algorithm \ref{alg_MPMLMC} without step \ref{alg_step_prec_choice}, where the computations in step \ref{alg_step_samples_comp} are carried out with an a-priori given, level-independent, and sufficiently high accuracy. In our numerical experiments we always describe precisely what accuracy we chose. 

\subsection{\rev{Application to the elliptic PDE problem}}\label{sec_MPMLMCFE}
In this section we show how the abstract analysis of computational error in MLMC from Section \ref{chap_MPMC} can be applied to the elliptic PDE problem. The precise problem statement is the following:

\begin{problem}\label{pr_qoi}
Let $G:\HH\rightarrow \R$ be a bounded linear functional and consider Problem \ref{pr_AVP_rnd}, an AVP 
with random data and solution $u(\cdot,\omega)\in \HH$ for a.e. $\omega \in \Omega$. We consider the problem of estimating the quantity of interest (QoI) defined as the expected value of the random variable $Q:\Omega\rightarrow\R$ given by $\omega \mapsto G(u(\cdot,\omega))$.
\end{problem}

Under assumptions \eqref{eq_ass_coeff_bound_rnd} and \eqref{eq_ass_higher_regularity_rnd}, it can be shown that when MLMC is applied to Problem \ref{pr_qoi}, we obtain \eqref{eq_mlmc_biasvarcost} with $\alpha = 2$ and $\beta = 4$. Generally $\gamma$ depends on the linear solver used and we discuss it in our numerical experiments. If an optimal linear solver is used (e.g., multigrid), one has $\gamma=2$; see \cite{cliffe2011multilevel}.

Throughout this section we will use the symbol $\u_h$ to denote the discrete solution of the AVP with random data (Problem \ref{pr_AVP_rnd}) expanded in the FE basis as in \eqref{eq_inexact_fe_sol}
with $\x$ computed effectively to precision $\ve$ such that
\begin{equation}\label{eq_eff_FEM_solution_linsys}
    \frac{\enorm{b(\omega)-A(\omega)\widehat{x}(\omega)}}{\enorm{b(\omega)}}\leq C\ve.
\end{equation}
The constant $C>0$ is independent of the problem data and $\omega$. The validity of this assumption in practice is discussed at the end of Section \ref{sec_MPMLMC_adapt}.

It can be shown that the MPML bias and variance decay assumptions \eqref{eq_mpmlbiasvardec_fem_general} are satisfied for Problem \ref{pr_qoi}. The corresponding values of the constants $\alpha_1$, $\alpha_2$, $\beta_1$, and $\beta_2$ are given by the following lemma.

\begin{lemma}\label{th_biasvar_dec_MPMLMC}
    Let $m>1$, and let $h_0,h_1,\ldots$ be discretisation parameters satisfying $h_0>0$ and $m h_l= 
    h_{l-1}$. Let $\u_{h_l}$ be the discrete solution of the AVP with random data (Problem \ref{pr_AVP_rnd}) computed effectively to precision $\ve_l$ on mesh level $l$, 
    and assume that there are $k_1,k_2>1$ such that $k_1\ve_l\leq\ve_{l-1}\leq k_2\ve_l$ for all $l\geq 1$. Then
    \begin{align}
    |\E[\qhel{l}-Q]|&=O(h_l^2+\ve_l),\label{eq_mpmlmcbiasdec4}\\
    \var[\yhel{l}]&=O(h_l^4+\ve_l^2).\label{eq_mpmlmcvardec4}
    \end{align}
\end{lemma}
\begin{proof}
    Using Jensen's inequality the bias error in \eqref{eq_mpmlmcbiasdec4} can be decomposed as 
    \begin{align}\label{aux23}
        |\E[\qhel{l}-Q]|&\leq \E[|G(\u_{h_l})-G(u)|]\nonumber\\
        &=\|G(u)-G(u_{h_l})+G(u_{h_l})-G(\u_{h_l})\|_{L^1(\Omega,\R)}\nonumber\\
        &\leq\|G(u)-G(u_{h_l})\|_{L^1(\Omega,\R)}+\|G(u_{h_l})-G(\u_{h_l})\|_{L^1(\Omega,\R)},
    \end{align}
    From \cite[Section 3]{cliffe2011multilevel} it follows that
    \begin{equation}\label{aux19}
        \|G(u)-G(u_h)\|_{L^1(\Omega,\R)}\leq Ch^2,
    \end{equation}
    where $C>0$ is independent of $h$, $u$, and $\omega$. 
    Since $G$ is bounded and $\u_h$ is computed effectively to precision $\ve$,
    it follows from Lemma \ref{th_fem_resiudal_est_rand} that
    \begin{equation*}
        \bigl|G(u_h(\cdot,\omega))-G(\u_h(\cdot,\omega))\bigr|\leq C\|f(\cdot,\omega)\|_{L^2(D)}\ve,
    \end{equation*}
    with a generic constant $C>0$ independent of $u$, $h$, and $\omega$.
    Integrating this inequality over $\Omega$ yields $\|G(u_h)-G(\u_h)\|_{L^1(\Omega,\R)}\leq C\ve$, which combined with \eqref{aux19} and \eqref{aux23} gives us the desired bound on the bias error in \eqref{eq_mpmlmcbiasdec4}.
    
    The variance bound in \eqref{eq_mpmlmcvardec4} can be shown similarly. Let us estimate
    \begin{align}\label{aux30}
         \var[\yhel{l}]&=\E[\yhel{l}^2]-\E[\yhel{l}]^2\\
         &\leq \E\bigl[\bigl(\qhel{l}-Q_{l}+Q_{l}-Q+Q-Q_{{l-1}}+Q_{{l-1}}-\qhel{l-1}\bigr)^2\bigr].\nonumber
    \end{align}
    Using the same technique as in \cite{cliffe2011multilevel}, we obtain an estimate of the form
    \begin{align}\label{aux31}
        \var[\yhel{l}]\leq &C\bigl(\E[(\qhel{l}-Q_{l})^2]+\E[(Q_{l}-Q)^2]\nonumber\\
        &+\E[(Q-Q_{{l-1}})^2]+\E[(Q_{{l-1}}-\qhel{l-1})^2]\bigr).
    \end{align}
    As in \cite{cliffe2011multilevel}, the quantity $\E[(Q_{l}-Q)^2]+\E[(Q-Q_{{l-1}})^2]$ can be estimated by
    \begin{equation}\label{aux32}
        \E[(Q_{l}-Q)^2]+\E[(Q-Q_{{l-1}})^2]\leq Ch_l^4.
    \end{equation}
    To bound the other two terms in \eqref{aux31} 
    we proceed as follows: As above, since  $G$ is bounded and $\u_{h_l}$ is computed effectively to precision $\ve_l$, Lemma \ref{th_fem_resiudal_est_rand} gives
    \begin{equation*}
       \bigl|G(u_{h_l}(\cdot,\omega))-G(\u_{h_l}(\cdot,\omega))\bigr|\leq C\|f(\cdot,\omega)\|_{L^2(D)}\ve_l,
   \end{equation*}
   for a generic constant $C>0$ independent of $u$, $h_l$, and $\omega$. Taking the second power of this inequality and integrating over $\Omega$ yields
   \begin{equation}\label{aux33}
       \|G(u_{h_l})-G(\u_{h_l})\|_{L^2(\Omega,\R)}^2\leq C\ve_l^2,
   \end{equation}
   Similarly, $\|G(u_{h_{l-1}})-G(\u_{h_{l-1}})\|_{L^2(\Omega,\R)}^2\leq C\ve_{l-1}^2\leq C k_2^2 \ve_l^2$.
   Together with \eqref{aux31},  \eqref{aux32} and \eqref{aux33} this yields the desired estimate \eqref{eq_mpmlmcvardec4}.
\end{proof}

This lemma allows us to formulate a specific version of Corollary \ref{cor:abstractMPMLMCcomplex_fem} regarding the error of the MPML estimator applied to Problem \ref{pr_qoi} discretised using finite elements. We summarize this in the following corollary.

\begin{corollary}[Error of the MPML FEM]\label{th_MPMLMCFE_complex}
    Let the assumptions of Lemma \ref{th_biasvar_dec_MPMLMC} be satisfied. Let $L\in\N$ and $N_0,\ldots,N_L\in\N$ and let $\qmpmlmc$ be the corresponding MPML estimator. Then the MSE of this estimator satisfies
\begin{equation*}
    \E\bigl[(\E[Q]-\qmpmlmc)^2\bigr]\leq C\biggl(\Bigl(h_L^{4}+\sum_{l=0}^L\frac{h_l^4}{N_l}\Bigr)+\Bigl(h_L^{2}\ve_L+\ve_L^2+\sum_{l=0}^L\frac{\ve_l^2}{N_l}\Bigr)\biggr).
\end{equation*}
\end{corollary}
\begin{proof}
    The claim follows from Lemma \ref{th_biasvar_dec_MPMLMC} and Corollary \ref{cor:abstractMPMLMCcomplex_fem}.
\end{proof}

    \rev{Lemma \ref{th_biasvar_dec_MPMLMC} now allows us to specify the accuracy parameters $\delta_l$ of the adaptive MPML algorithm from Section \ref{sec_MPMLMC_adapt} for Problem \ref{pr_qoi}. Using Lemma \ref{th_biasvar_dec_MPMLMC} in \eqref{eq_prec_choice} we get
    \begin{equation}\label{eq_prec_choice_fem}
        \begin{split}
            \ve_l&\coloneqq \sqrt{k_p}h_l^{2},\quad l=0,\ldots,L-1,\\
            \ve_L&\coloneqq k_ph_L^2.
        \end{split}
    \end{equation}}
%
 
    \rev{To quantify the cost gains, let us apply the abstract analysis from Section \ref{sec_cost_analysis} to Problem \ref{pr_qoi}. Using an optimal linear solver (e.g., multigrid), we have $\var[Y_l]=O(2^{-4l})$ and $C_l=O(2^{2l})$ (see \cite{cliffe2011multilevel}). If in \eqref{eq_cost_saving_ass} we have $q=1/4$, for example, then \mbox{Corollary \ref{cor_mpml_mlmc_cost_savings}} predicts that the cost is reduced by at least a factor of 1.6.} 
    
    In our experiments, using the MINRES method, we observed an actual cost gain in terms of floating point operations by a factor of $\approx1.5$ (see Figure \ref{fig_mpml_minres_cost}). When a low precision sparse direct solver with iterative refinement is used, the reduction in allocated memory is reduced by up to a factor of $\approx3.5$ (see Figure \ref{fig_memory_cost_cholesky_simple}).

\begin{remark}[\rev{Application to other sampling methods}]\label{rem:other_samplers}
    \rev{Note that the extension of the standard MLMC assumptions \eqref{eq_mlmc_biasvarcost} to the more general MPML error model \eqref{eq_mpmlmcbiasdec3} and \eqref{eq_mpmlmcvardec3} is not specific to multilevel Monte Carlo. Also, Lemma \ref{th_biasvar_dec_MPMLMC} is not specific to the sampling method. It is a statement about the quantity of interest.} 
    
    \rev{In fact, assumptions of the form \eqref{eq_mlmc_biasvarcost} are central to the analysis of many other multilevel sampling methods, such as multilevel MCMC \cite{dodwell2019multilevel}. The assumptions \cite[Theorem 3.4, Assumptions M1., M2., M3.]{dodwell2019multilevel} for the analysis of multilevel MCMC are analogous to \eqref{eq_mlmc_biasvarcost} and the method is also tested on a similar model problem;
see \cite[eq. (4.2)]{dodwell2019multilevel}. They can be extended in the same way as \eqref{eq_mlmc_biasvarcost} to \eqref{eq_mpmlmcbiasdec4} and \eqref{eq_mpmlmcvardec4} to obtain a performance gain in multilevel MCMC via inexact computations and mixed precision arithmetic. As a second example, consider multi-index Monte Carlo \cite{haji2016multi}. An analogy to \eqref{eq_mlmc_biasvarcost} for MIMC is \cite[Assumptions 1 to 3]{haji2016multi}. An extension of these assumptions to take into account computational error is again straightforward.}
\end{remark}

\section{Numerical results}\label{chap_num_results}

We provide numerical experiments demonstrating the potential cost savings of the adaptive MPML algorithm (Algorithm \ref{alg_MPMLMC}) over a standard adaptive MLMC algorithm. The adaptive MPML algorithm preserves the overall computational accuracy. As a suitable model, we use an elliptic PDE with lognormal random coefficients with both a direct and an iterative linear solver. 

\rev{In principle, any suitable iterative solver with the stopping criterion given by \eqref{eq_eff_FEM_solution_linsys} can be used
to compute the approximate discrete solution $\u_{h_l}$ on the level $l$ effectively to precision $\ve_l$. Since the values of~$\ve_l$ obtained using the adaptive MPML algorithm are typically relatively ``big'' (see Section \ref{sec_num_res_ellipticpde} for examples), the iterative solver can potentially achieve the tolerance in a very small number of iterations, leading to significant cost gains. As an example, we employ MINRES as described in Section \ref{sec_kryl_methods}.
In this case, the cost gains do not come primarily from using low precision, but rather from reducing the number of iterations. 
A technique, where the cost gains come from the use of low precision, is iterative refinement. It can in principle be used with any direct or iterative solver; see Section \ref{sec_itref} and \cite{carson2018accelerating}. Here, we used it in combination with a direct solver based on Cholesky factorisation.}

All numerical experiments are implemented in Python. The codes are available at \mbox{\url{https://github.com/josef-martinek/mpml}}.

\subsection{Elliptic PDE with lognormal coefficients -- iterative solver}\label{sec_num_res_ellipticpde}

We will solve an equation of the following form, which is a special case of \eqref{eq_avp_rnd}:
\begin{align}\label{eq_pde_lognormal_coef}
    \begin{split}
    -\nabla\cdot\bigl(a(\cdot,\omega)\nabla u(\cdot,\omega)\bigl)&=f\quad \text{on }D,\\
    u(\cdot,\omega)&=0\quad \text{on } \partial D,
    \end{split}
\end{align}
for a given random field $a$ and a deterministic right-hand side $f$. The random field is chosen in such a way that it corresponds to a truncated Karhunen-Loève expansion of a suitable covariance operator, in particular,
\begin{equation}\label{eq_rf_logn}
    a(x_1,x_2,\omega)=\exp{\biggl(\sum_{j=1}^s\omega_j\frac{1}{j^q}\sin{(2\pi jx_1)}\cos{(2\pi jx_2)}\biggr)}.
\end{equation}
Here $\omega=(\omega_1,\ldots,\omega_s)\in\R^s$ is such that $\omega_j\sim N(0,\sigma^2)$ for a fixed $\sigma>0$. Random fields of this form are widely used; see \cite{babuska2004galerkin}, \cite{nobile2008sparse}, and \cite{cliffe2011multilevel} for examples.

As the quantity of interest we choose
\begin{equation*}
    \E[Q]=\int_\Omega \Bigl(\int_D u(x_1,x_2,\omega)\text{d}(x_1,x_2)\Bigr)\text{d}\omega.
\end{equation*}

Due to the fact that $\omega_j\sim N(0,\sigma^2)$, assumptions \eqref{eq_ass_coeff_bound_rnd} and \eqref{eq_ass_higher_regularity_rnd} are not satisfied. By choosing, e.g., $\omega_j\sim \text{Uni}(0,c)$, one could ensure that both assumptions are satisfied, but we want to test the developed methods in a more realistic setting.
The analysis in Section \ref{sec_fem_rand_fin_prec} could easily be extended to this setting, using the analysis for MLMC presented in \cite{teckentrup2013further}, but we did not present this here to avoid unnecessary technicalities.

\begin{table}[t]
    \centering
    \begin{tabular}{|r|r r r r r|}
		\cline{2-6}
		\multicolumn{1}{c|}{} & \multicolumn{5}{c|}{$k_p=0.05$} \\
		\hline
		\multicolumn{1}{|c|}{$L$} & \multicolumn{1}{c|}{$\ve_0$} & \multicolumn{1}{c|}{$\ve_1$} & \multicolumn{1}{c|}{$\ve_2$} & \multicolumn{1}{c|}{$\ve_3$} & \multicolumn{1}{c|}{$\ve_4$} \\
		\hline
		$1$ & $3.5$e$-3$ & $1.9$e$-4$ & - & - & - \\
        \hline
		$2$ & $3.5$e$-3$ & $8.7$e$-4$ & $4.9$e$-5$ & - & - \\
		\hline
        $3$ & $3.5$e$-3$ & $8.7$e$-4$ & $2.2$e$-4$ & $1.2$e$-5$ & - \\
		\hline
        $4$ & $3.5$e$-3$ & $8.7$e$-4$ & $2.2$e$-4$ & $5.5$e$-5$ & $3.1$e$-6$ \\
		\hline
	\end{tabular}
    \quad
    \begin{tabular}{|c|c|}
        \hline
         $l$ & precision values \\
         \hline
         $0$ & $hhss$ \\
         \hline
         $1$ & $ssss$ \\
         \hline
         $2+$ & $ssdd$ \\
         \hline
    \end{tabular}
    \caption{Left: Required effective precision on each level (in terms of relative residual), determined by the adaptive MPML algorithm (Algorithm \ref{alg_MPMLMC}) for different values of the finest level $L$. \ Right: Choice of precision values $(\ur_f,\ur_7,\ur,\ur_r)$ for iterative refinement on each level $l$. The factorisation is carried out in half precision on level $0$.}
    \label{tab_Test2_MPMLMC_setup_k0.1}
\end{table}

In the first example of this section, we choose the data in \eqref{eq_pde_lognormal_coef} as follows. The right-hand side satisfies $f\equiv 1$ \rev{on $D=(0,1)^2$} and the parameters in the coefficient function are chosen as $s=4$, $q=2$, and $\sigma=2$. To solve this problem numerically, we discretise~\eqref{eq_pde_lognormal_coef} using the FEM as described in Section \ref{sec_fem_rand} with simplical elements implemented in FEniCSx \cite{baratta_2023_10447666}. In this experiment, the discretisation parameter on the coarsest mesh is $h_0=1/8$ and the mesh is refined on each level by a factor $m=2$.

Recall that in all our experiments we refer to standard MLMC as Algorithm \ref{alg_MPMLMC} without step \ref{alg_step_prec_choice}, where the computations in step \ref{alg_step_samples_comp} are carried out with an a-priori given, sufficiently high accuracy. In each experiment we always describe precisely what this accuracy is. For MLMC, the parameters $\alpha$ and $\beta$ in the MLMC complexity theorem \cite[Theorem 1]{cliffe2011multilevel} (see also \eqref{eq_mlmc_biasvarcost}) can be chosen to be $\beta=4$, $\alpha=2$ for the random field in \eqref{eq_rf_logn}. To set up the MPML algorithm (Algorithm \ref{alg_MPMLMC}), we further need the parameters $\alpha_1$, $\alpha_2$, $\beta_1$, and $\beta_2$ from the MPML complexity theorem (Corollary \ref{cor:abstractMPMLMCcomplex_fem}). As shown in Lemma \ref{th_biasvar_dec_MPMLMC}, we have $\alpha_1=\alpha=2$, $\beta_1=\beta=4$, $\alpha_2=1$, and $\beta_2=2$, \rev{which results in the effective precision choice given by \eqref{eq_prec_choice_fem}}. In both 
adaptive algorithms, the underlying linear systems are solved using the PETSc implementation of MINRES \cite{petsc_4_py} with a stopping criterion given by the relative residual norm. The tolerances for the stopping criterion in the MPML algorithm are given in Table \ref{tab_Test2_MPMLMC_setup_k0.1} (left). The standard MLMC algorithm uses a fixed tolerance specified below for each experiment and all computations are carried out in double precision without iterative refinement. The cost of solving the linear system is estimated in terms of the number of floating point operations (FLOPs) performed by MINRES. To count the FLOPs
we use the PETSc GetFlops() function.

\begin{figure}[t]
         \centering
         \includegraphics[scale=0.35]{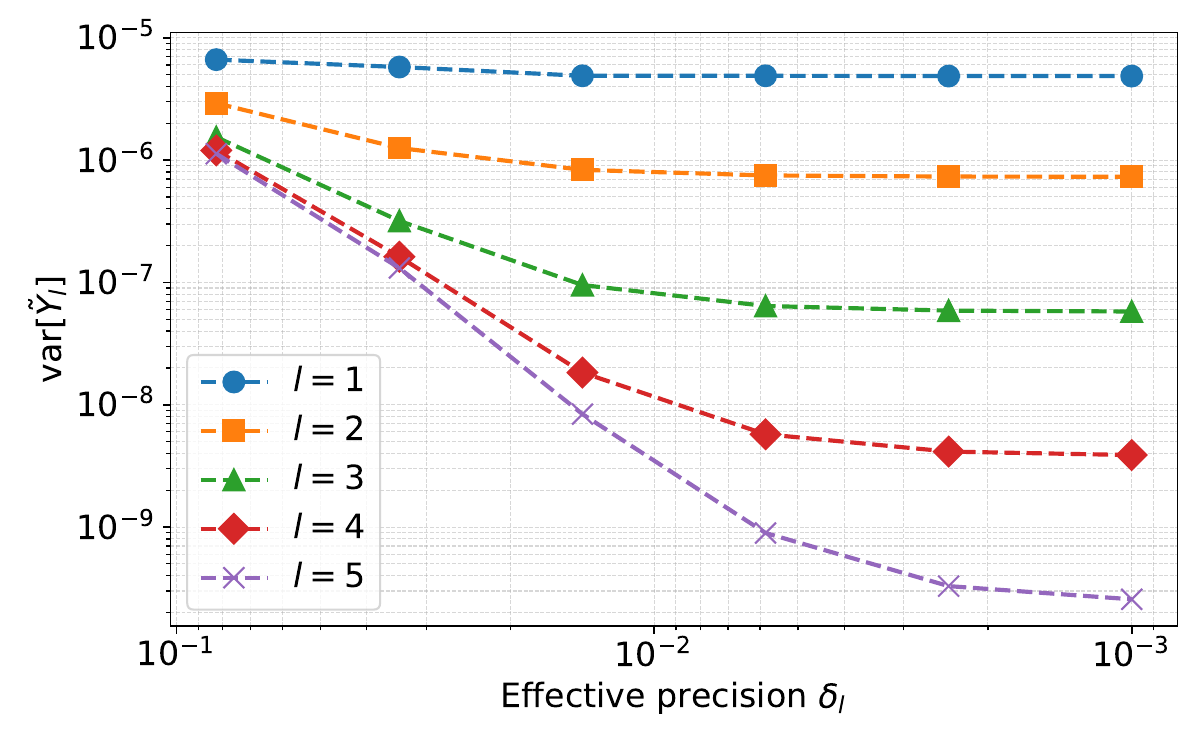}
     \hspace{2mm}
         \includegraphics[scale=0.35]{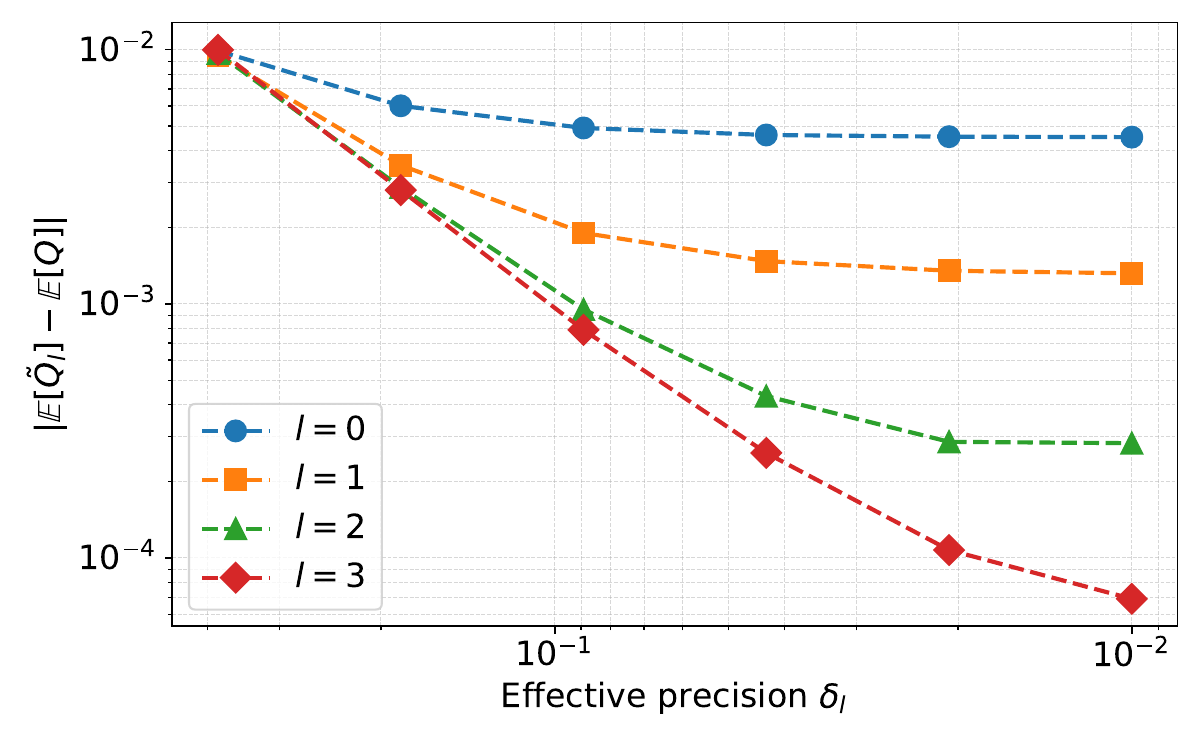}
        \caption{\rev{Computational error in the variance of the difference estimators $\var[\Tilde{Y}_{l}]$ (left) and in the bias $|\E[\Tilde{Q}_{l}-Q]|$ (right), plotted against effective precision. Discretisation level $l$ corresponds to $h_l=1/8\times2^{-l}$. To achieve an effective precision $\ve_l$, relative residuals produced by MINRES are monitored.}}
        \label{fig_biasvardec}
\end{figure}

To determine the effective precision $\ve_l$ in the MPML adaptive algorithm, we use formula \eqref{eq_prec_choice} 
with the constant $k_p=0.05$. As discussed below and depicted in Figure \ref{fig_mpml_minres_kp_mse}, the MPML algorithm is not sensitive with respect to the choice of $k_p$. For a multilevel estimator with a given number of levels, the choice of $k_p=0.05$ then determines uniquely the effective precisions $\ve_l$ on all levels according to \eqref{eq_prec_choice}. The values of $\ve_l$ for different estimators can be found in Table \ref{tab_Test2_MPMLMC_setup_k0.1} (left).

\rev{We begin by visualising computational error in the variance $\var[\Tilde{Y}_{l}]$ and in the bias $|\E[\Tilde{Q}_{l}-Q]|$ (Figure \ref{fig_biasvardec}). We plot both variance and bias against effective precision $\ve_l$. 
To achieve a given value of the effective precision $\ve_l$, we monitor the  relative residuals produced by the MINRES method.
A bound for the computational error in these quantities has been given in Lemma \ref{th_biasvar_dec_MPMLMC}.}

\rev{In Figure \ref{fig_biasvardec}, we observe that the discretisation error quickly begins to dominate the variance and the bias on each level, as the computational error decays. On the fine levels, the variance decay rate (left) and the bias decay rate (right) become apparent. The initial decay rate is at least $O(\ve_l^2)$ for the variance and at least $O(\ve_l)$ for the bias, which are the rates obtained in Lemma \ref{th_biasvar_dec_MPMLMC}. 
As the discretisation error starts to dominate, i.e., for small values of $\ve_l$ and large values of $h_l$, standard multilevel variance and bias are recovered. The sample variances have been produced using $10^2$ samples each, while for the bias estimates $10^4$ samples were used.}

\rev{Let us note that the results in Figure \ref{fig_biasvardec} are not specific to the multilevel Monte Carlo method, only to the QOI. This shows the potential for using inexact computations in a broad spectrum of multilevel sampling methods if the QOI is sufficiently well-behaved, as discussed in Remark \ref{rem:other_samplers}.}
\begin{figure}[t]
\begin{minipage}[t]{0.48\textwidth}
    \centering
    \vspace{0pt} 
        \includegraphics[scale=0.35]{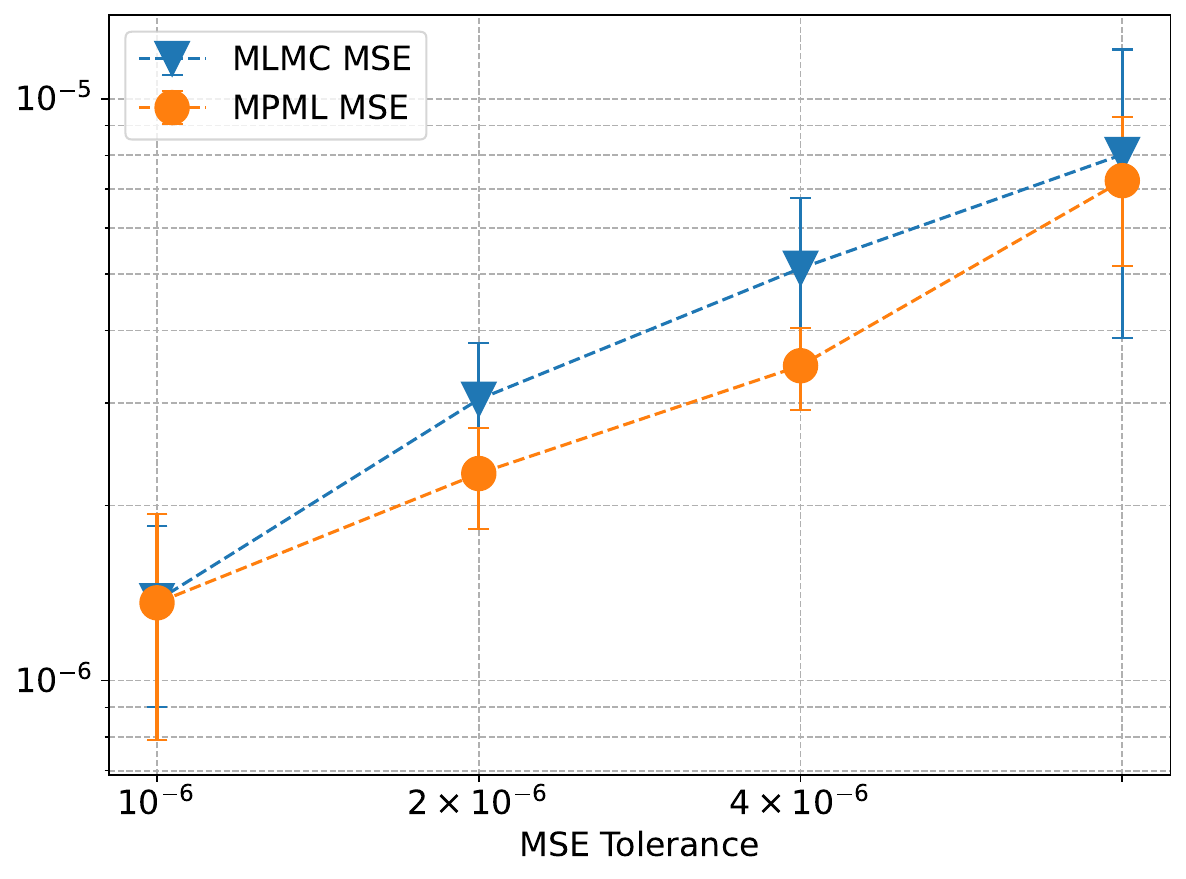}
        \captionof{figure}{Estimated MSE with approximate $95\%$ confidence intervals for the MPMC estimator when compared to the reference MLMC estimator for various target tolerances. As their linear solver, both estimators use MINRES; MPML uses an adaptively chosen stopping criterion.}
        \label{fig_mpml_minres_mse}
\end{minipage}
\hspace{2mm}
\begin{minipage}[t]{0.48\textwidth}
    \centering
    \vspace{0pt}
        \includegraphics[scale=0.35]{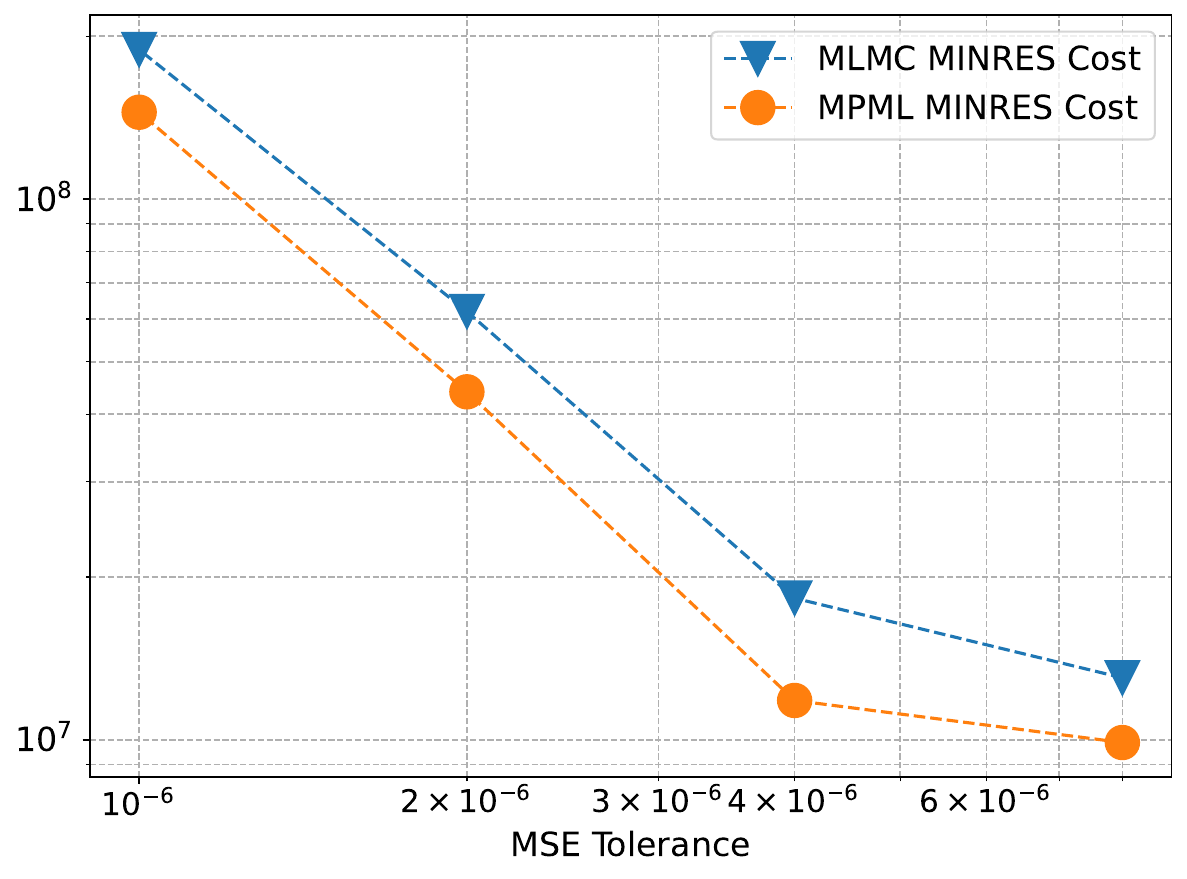}
        \captionof{figure}{Total cost gain in terms of FLOPs for various tolerances using adaptive MPML compared to the standard adaptive MLMC estimator. As their linear solver, both estimators use MINRES; MPML uses an adaptively chosen stopping criterion.}
        \label{fig_mpml_minres_cost}
\end{minipage}
\end{figure}
\begin{figure}[t!]
\begin{minipage}[t]{0.48\textwidth}
    \centering
    \vspace{0pt} 
        \includegraphics[scale=0.35]{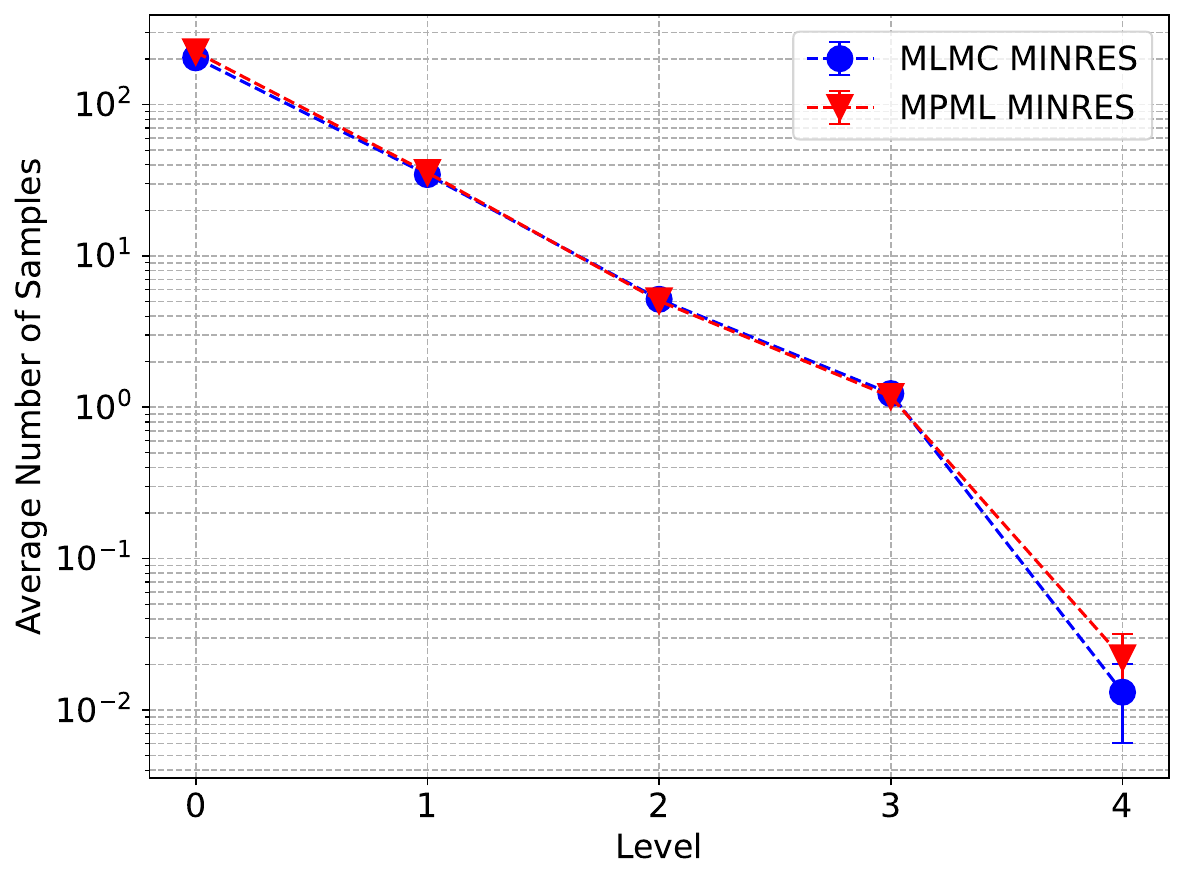}
        \captionof{figure}{Average number of samples on each level for adaptive MLMC and adaptive MPML for $\text{\text{TOL}}^2=10^{-6}$ (with approximate $95\%$ confidence intervals).}
        \label{fig_mpml_minres_nsamp}
\end{minipage}
\hspace{2mm}
\begin{minipage}[t]{0.48\textwidth}
    \centering
    \vspace{0pt}
        \includegraphics[scale=0.35]{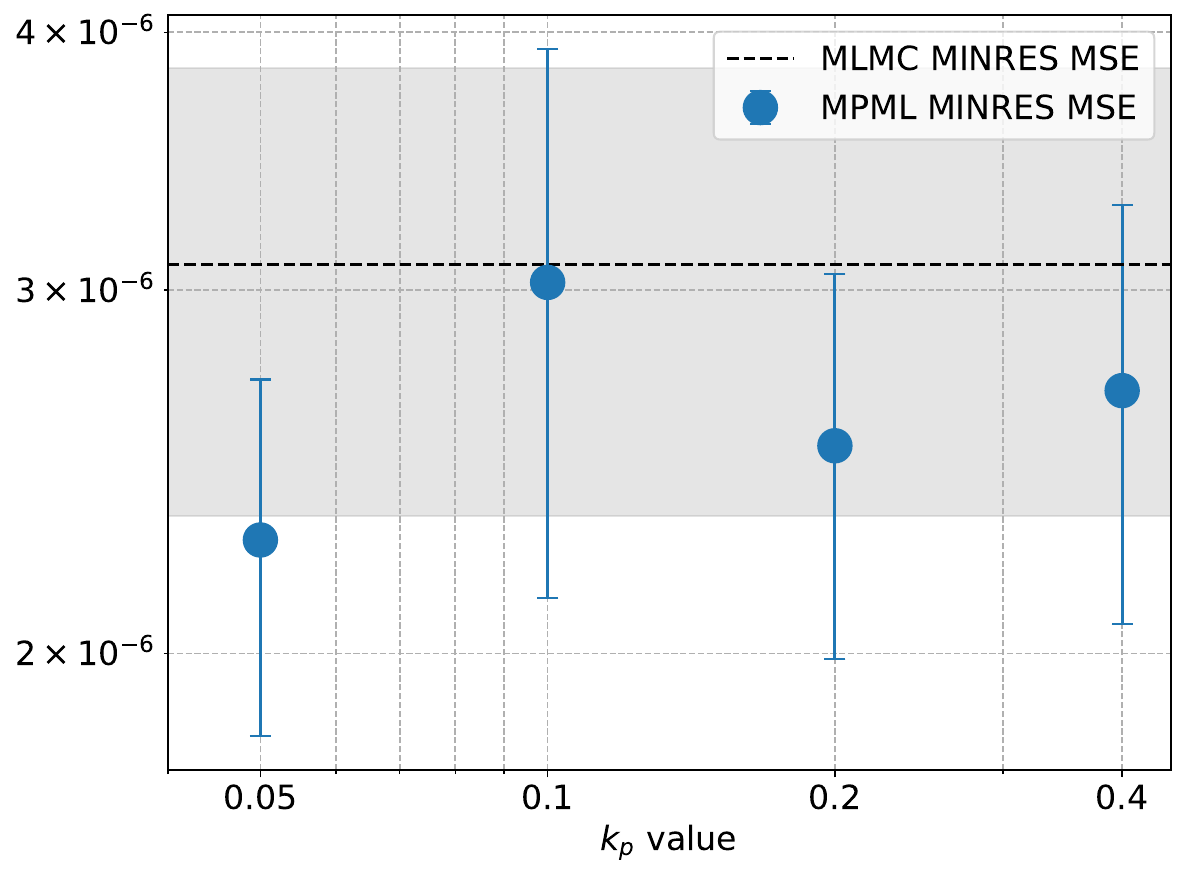}
        \captionof{figure}{MSE of the MPML algorithm for various values of $k_p$ compared to the reference MSE of MLMC (with approximate $95\%$ confidence intervals).}
        \label{fig_mpml_minres_kp_mse}
\end{minipage}
\end{figure}

We now aim to verify the accuracy of the MPML algorithm. We execute $1000$ runs of the adaptive MPML algorithm (Algorithm \ref{alg_MPMLMC}) for each of the four values of the MSE tolerance $\text{TOL}^2$, namely $\text{TOL}^2=8,4,2,1\times 10^{-6}$. For the same tolerances we perform $1000$ runs of the standard MLMC algorithm. The standard MLMC algorithm uses a fixed tolerance of $10^{-6}$ for the MINRES stopping criterion. Figure~\ref{fig_mpml_minres_mse} shows the MSE of the estimates obtained by both algorithms averaged over the number of runs with approximate $95\%$ confidence intervals. To estimate the MSE of both algorithms a reference value of the QoI is computed by the standard MLMC algorithm with tolerance $\text{TOL}^2=2\times10^{-8}$ and a direct, double precision linear solver. Up to statistical errors, both estimators have a similar MSE. However, they do differ slightly, due to the discrete choices in the two adaptive procedures, e.g., of the total number of levels $L$. 

In Figure \ref{fig_mpml_minres_cost}, we demonstrate the cost gains obtained by MPML (Algorithm \ref{alg_MPMLMC}) compared to standard MLMC, again with a fixed tolerance of $10^{-6}$ in the MINRES stopping criterion.
We plot the average cost in terms of FLOPs 
for each tolerance and
observe that it is consistently reduced by a factor of $\approx1.5$. As demonstrated in Figure \ref{fig_mpml_minres_mse}, this cost gain comes with no statistically significant loss of accuracy. It results from the fact that the effective precision $\ve_l$ for solving the linear systems on the coarser levels is chosen significantly larger than $10^{-6}$ in MPML,  see Table \ref{tab_Test2_MPMLMC_setup_k0.1} (left).

At the beginning of the cost analysis in Section \ref{sec_cost_analysis} we made the assumption that the numbers of samples on each level are chosen approximately the same within the MPML and the MLMC adaptive procedures. We verify numerically that this is true. Figure \ref{fig_mpml_minres_nsamp} shows the average number of samples obtained by MPML (Algorithm \ref{alg_MPMLMC}) and standard MLMC on each level for the tolerance $\text{TOL}^2=10^{-6}$. We observe that the number of samples are equal within statistical error given by the approximate $95\%$ confidence intervals. For this experiment, the standard MLMC algorithm uses a~fixed tolerance of $10^{-10}$ in the MINRES stopping criterion. \rev{Note that due to the fact that the finest level is not specified a-priori in Algorithm \ref{alg_MPMLMC} and due to the stochastic nature of the discrete choices in Algorithm \ref{alg_MPMLMC}, Level 4 is only reached in about 2 \% of all runs. This suggests that in this example Level 4 might not be needed in practice.} Note also that we need on average only one sample on Level 3.

In Section \ref{sec_MPMLMC_adapt} we claimed that the MPML algorithm is not sensitive to the a-priori chosen constant $k_p$ used to determine the required computational accuracy. Figure~\ref{fig_mpml_minres_kp_mse} shows the average MSE of the MPML algorithm (Algorithm \ref{alg_MPMLMC}) for the MSE tolerance $\text{TOL}^2=2\times10^{-6}$ with approximate $95\%$ confidence intervals. The MSE is estimated using $1000$ runs of the algorithm for the values $k_p=0.05, 0.1, 0.2, 0.4$. The estimated MSE of the standard MLMC algorithm is shown for comparison with a dashed line. It is estimated using $1000$ runs of standard MLMC with a tolerance of $10^{-10}$ in the MINRES stopping criterion. The choice of the constant $k_p$ seems to have no significant impact on the overall accuracy, provided it is sufficiently small. For our suggested choice of $k_p=0.05$ the computational error in our model problem is bounded safely.

\subsection{Elliptic PDE with lognormal coefficients -- direct solver}

In this section, we again solve \eqref{eq_pde_lognormal_coef} with the same input data, i.e., $f\equiv 1$ \rev{on $D=(0,1)^2$}, $s=4$, $q=2$, and $\sigma=2$. The coarsest mesh size is again $h_0=1/8$. However, here the underlying linear system (with symmetric positive definite matrix) is solved using a double precision Cholesky factorisation from PETSc in the standard MLMC algorithm and a low precision Cholesky factorisation with iterative refinement in our MPML algorithm. Recall that throughout our experiments by standard MLMC we mean Algorithm~\ref{alg_MPMLMC} without step \ref{alg_step_prec_choice}, where the computations in step \ref{alg_step_samples_comp} are (here) carried out with double precision. We use our own implementation of iterative refinement and of the low precision Cholesky factorisation. To carry out computations in half, single, and double precision, we use the numerical types of NumPy, namely float16(), float32(), and float64() for half, single, and double precision, respectively. Apart from the linear solver, all calculations are carried out in double precision.

The iterative refinement (Algorithm \ref{alg_itref}) is set up as follows. As in the previous section, the actual, numerical values in the stopping criterion are given by the required effective precisions $\ve_l$ on each level, specified by formula \eqref{eq_prec_choice}. To this end, we also choose again $\alpha_1=\alpha=2$, $\beta_1=\beta=4$, $\alpha_2=1$, and $\beta_2=2$. Table \ref{tab_Test2_MPMLMC_setup_k0.1} (left) shows the resulting values of the effective precision $\ve_l$.
Table \ref{tab_Test2_MPMLMC_setup_k0.1} (right) shows the exact setting of the other precisions used in the iterative refinement algorithm. The algorithm contains $3$ precisions, i.e., $\ur_f$, $\ur$, and $\ur_r$, each taking on one of the values quarter~($q$), half ($h$), single ($s$), or double ($d$). To simplify the notation, we describe the exact setting of the iterative refinement schematically as an ordered quadruple, e.g., $(\ur_f,\ur_8,\ur,\ur_r)=(hhss)$, where $\ur_8$ is the precision chosen at step \ref{aux27} of Algorithm~\ref{alg_itref}.

In the cost model~\eqref{eq_mlmc_biasvarcost} we choose $\gamma=2$. In 2D, this corresponds to the case when the linear system is solved in linear complexity, since the number of unknowns grows quadratically with the mesh size.

We start again by determining the accuracy of the MPML estimator $\qmpmlmc$, using the setting as described above, and we compare it to the accuracy of the standard MLMC estimator $\qmlmc$ with double precision Cholesky as the linear solver on all levels. We perform $1000$ runs of the MPML estimator and of the MLMC estimator for each of the MSE tolerances $\text{TOL}^2=8,4,2,1\times 10^{-6}$. To eliminate randomness from estimating the MSE and to get a better idea of how big the computational error actually is, we use the same numbers of samples in the MPML and MLMC estimators and the same seeds in the random number generators. The number of samples $\{N_l\}$ we use for both estimators is determined by $1000$ runs of the standard adaptive MLMC algorithm. Figure \ref{fig_mse_cholesky} shows the resulting average mean squared errors of the MPML and MLMC estimators. The difference in the MSEs of both estimators is very small; in fact the relative error of the MPML MSE with respect to the MLMC MSE is less than $0.01\%$. Recall that the MPML estimator uses half and single precision for the matrix factorisation on all levels; see Table \ref{tab_Test2_MPMLMC_setup_k0.1} (right). This promises significant memory savings when implemented efficiently on an architecture where half precision computations are supported.

We continue by estimating the cost gains using MPML estimator $\qmpmlmc$ with low precision Cholesky factorisation and iterative refinement compared to standard MLMC estimator $\qmlmc$ with double precision Cholesky. Again, both estimators use the same numbers of samples $\{N_l\}$. Since memory references are by far the most costly part on modern computing architectures, both in terms of time and energy cost (cf.~the discussion Section \ref{sec:energy}), we use memory access as a simple cost model, in terms of total number of bits loaded into main memory.
Thus, accessing a half precision floating point number is $2\times$ cheaper than accessing a single precision number and $4\times$ cheaper than accessing a double precision number. The simulated cost gain can then be estimated by counting the number of entries in all vectors and the number of non-zeros in all sparse factorisations of the matrices computed and stored in the process of solving the linear systems during the iterative refinement process. Figure \ref{fig_memory_cost_cholesky} shows the total cost gain per level (cost per sample $\times$ number of samples) for $\qmpmlmc$ compared to $\qmlmc$. On all levels we observe a simulated memory gain of $\approx2$. The total memory gain (sum over all levels) is $\approx 2$ as well.

In the last example, we show how further memory gains can be obtained using quarter precision on the coarsest level. For this experiment, we again solve \eqref{eq_pde_lognormal_coef} with $f\equiv1$ on $D$, $s=1$ and $\sigma=1$, but now we choose the coarsest mesh size to be \mbox{$h_0=1/4$}. To determine the effective precision $\ve_l$ in \eqref{eq_prec_choice} we choose $k_p=0.4$. Iterative refinement with the Cholesky solver is used with the following settings: $qhhh$ on level~$0$, and $hhss$ on levels $1$ to $3$. This means that on level $0$ where \mbox{$h_0=1/4$}, quarter precision (q43, see Section \ref{sec_fp}) is used for the Cholesky factorisation. To simulate quarter precision, the {\tt pychop} package of Xinye Chen was used; see \cite{pychop}. Figures~\ref{fig_mse_cholesky_simple}~and~\ref{fig_memory_cost_cholesky_simple} show the MSE for MPML and MLMC and the simulated memory gain, respectively. They have been produced analogously to Figures~\ref{fig_mse_cholesky}~and~\ref{fig_memory_cost_cholesky}. We observe that for the MSE tolerance $2\times10^{-6}$, the MPML relative error is $11\%$ bigger compared to standard MLMC error with the same number of samples, while we are able to achieve a total memory gain of $\approx3.5$. Recall that the cost gain is simulated by counting non-zeros computed and stored in the process of solving the linear system~$Ax=b$.

We have not encountered overflow in any of the examples presented in this section when working with the lower precisions. It is important to note that due to the fact that the coefficients of the PDE are lognormally distributed, overflow can occur with low probability. In the case where overflow occurs, scaling or shifting techniques can be used; see \cite{higham2019squeezing}.
\begin{figure}[t]
\begin{minipage}[t]{0.48\textwidth}
    \centering
    \vspace{0pt} 
        \includegraphics[scale=0.35]{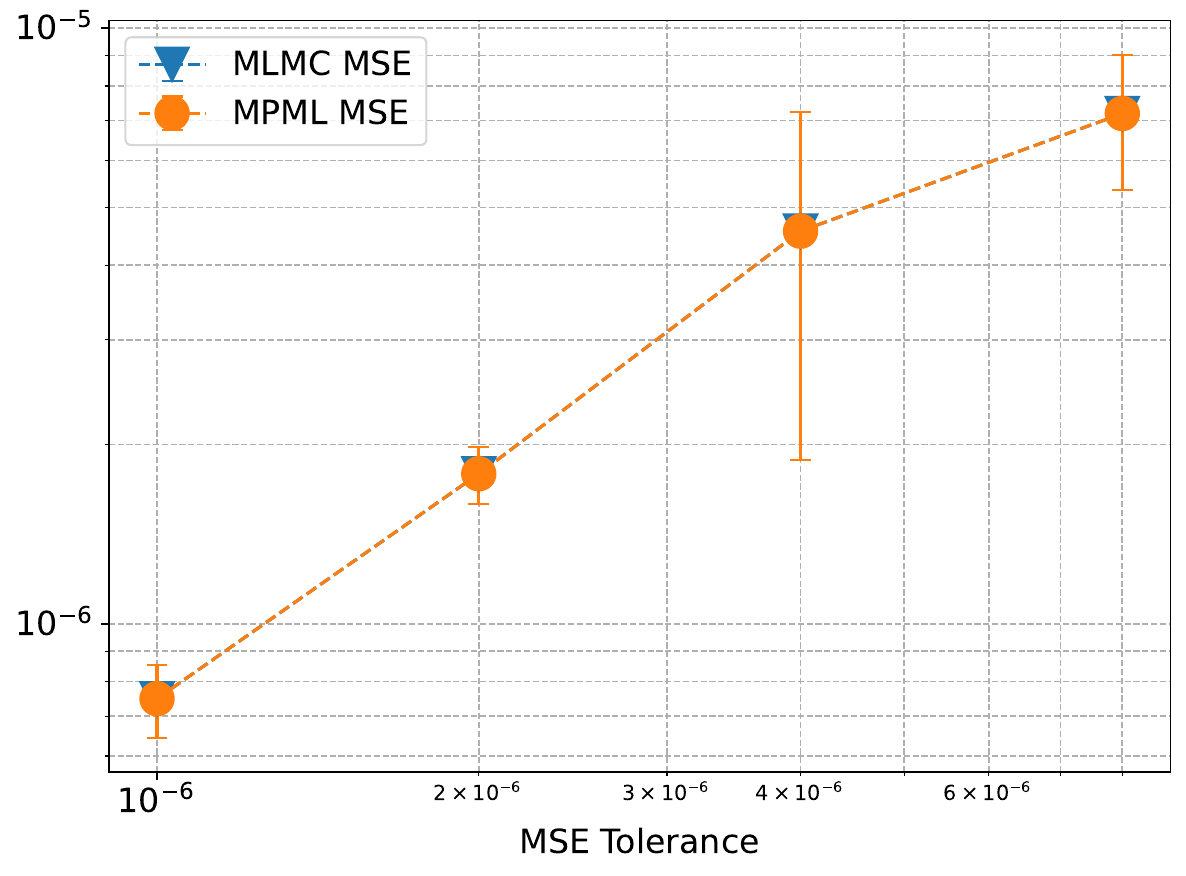}
        \captionof{figure}{Comparison of estimated MSE for MPML and  MLMC using the same number of samples and the same random seeds (with approximate $95\%$ confidence intervals). MLMC uses double precision Cholesky, while MPML uses low precision Cholesky with iterative refinement.}
        \label{fig_mse_cholesky}
\end{minipage}
\hspace{2mm}
\begin{minipage}[t]{0.48\textwidth}
    \centering
    \vspace{0pt}
        \includegraphics[scale=0.35]{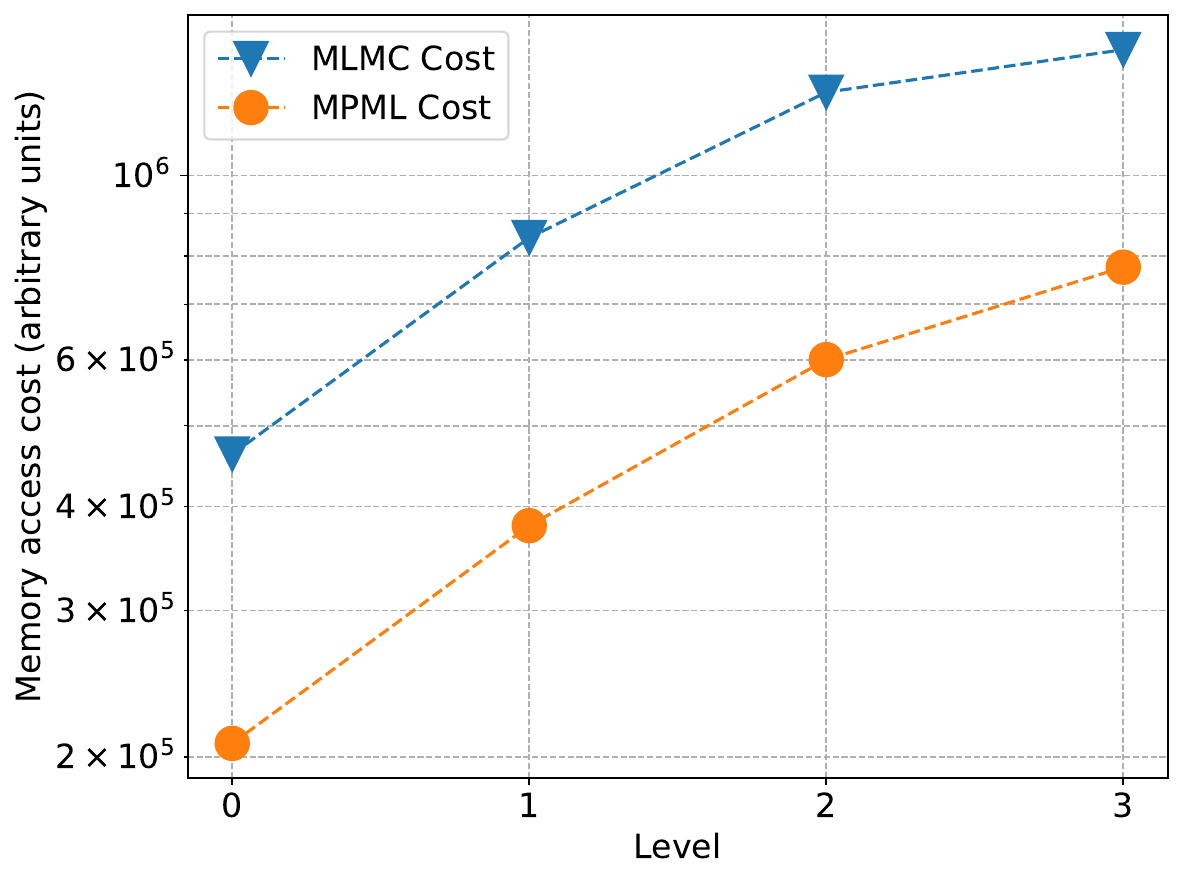}
        \captionof{figure}{Comparison of total costs of the estimators per level in terms of memory access for MPML  and MLMC. MLMC uses double precision Cholesky, while MPML uses low precision Cholesky with iterative refinement.}
        \label{fig_memory_cost_cholesky}
\end{minipage}
\end{figure}
\begin{figure}[t]
\begin{minipage}[t]{0.48\textwidth}
    \centering
    \vspace{0pt} 
        \includegraphics[scale=0.35]{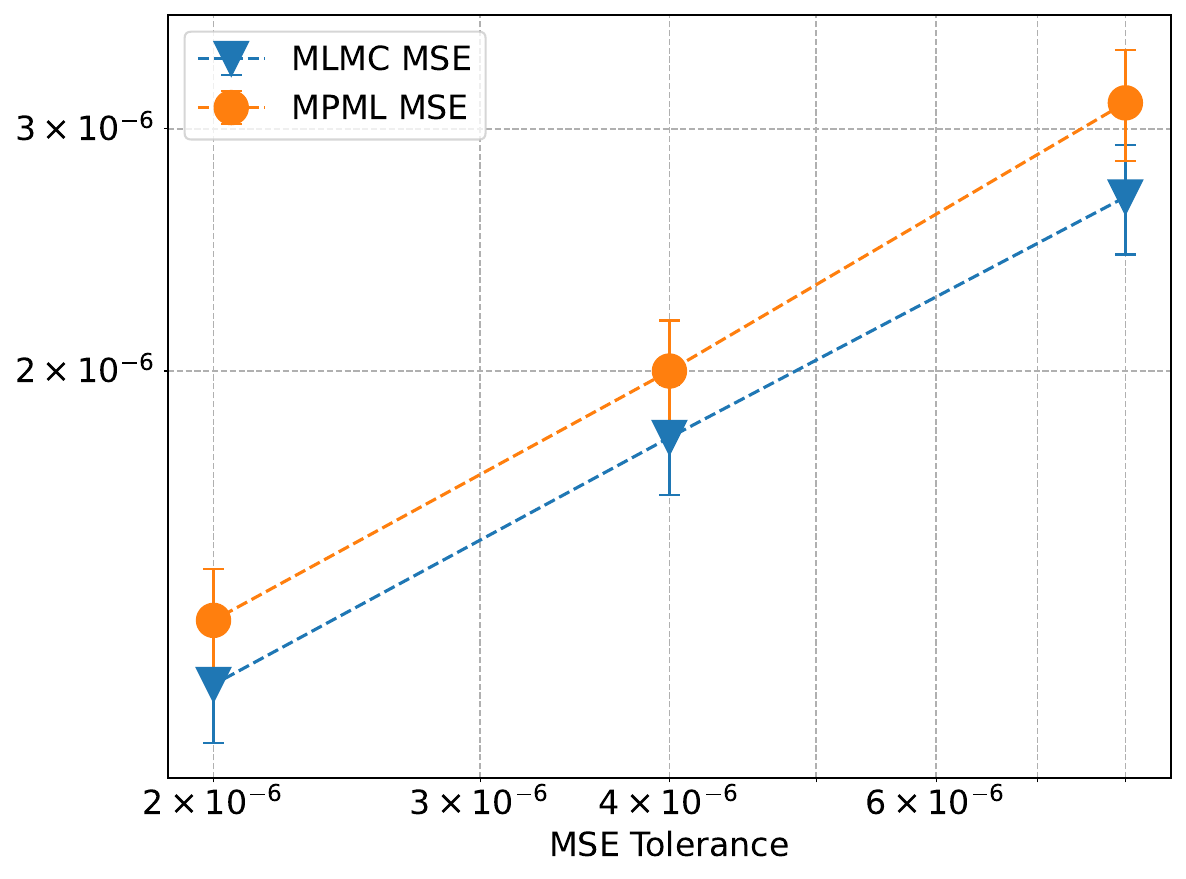}
        \captionof{figure}{Comparison of estimated MSE for MPML and  MLMC using the same number of samples and the same random seeds (with approximate $95\%$ confidence intervals). MLMC uses double precision Cholesky, while MPML uses low precision Cholesky with iterative refinement (here with quarter precision on level $0$).}
        \label{fig_mse_cholesky_simple}
\end{minipage}
\hspace{2mm}
\begin{minipage}[t]{0.48\textwidth}
    \centering
    \vspace{0pt}
        \includegraphics[scale=0.35]{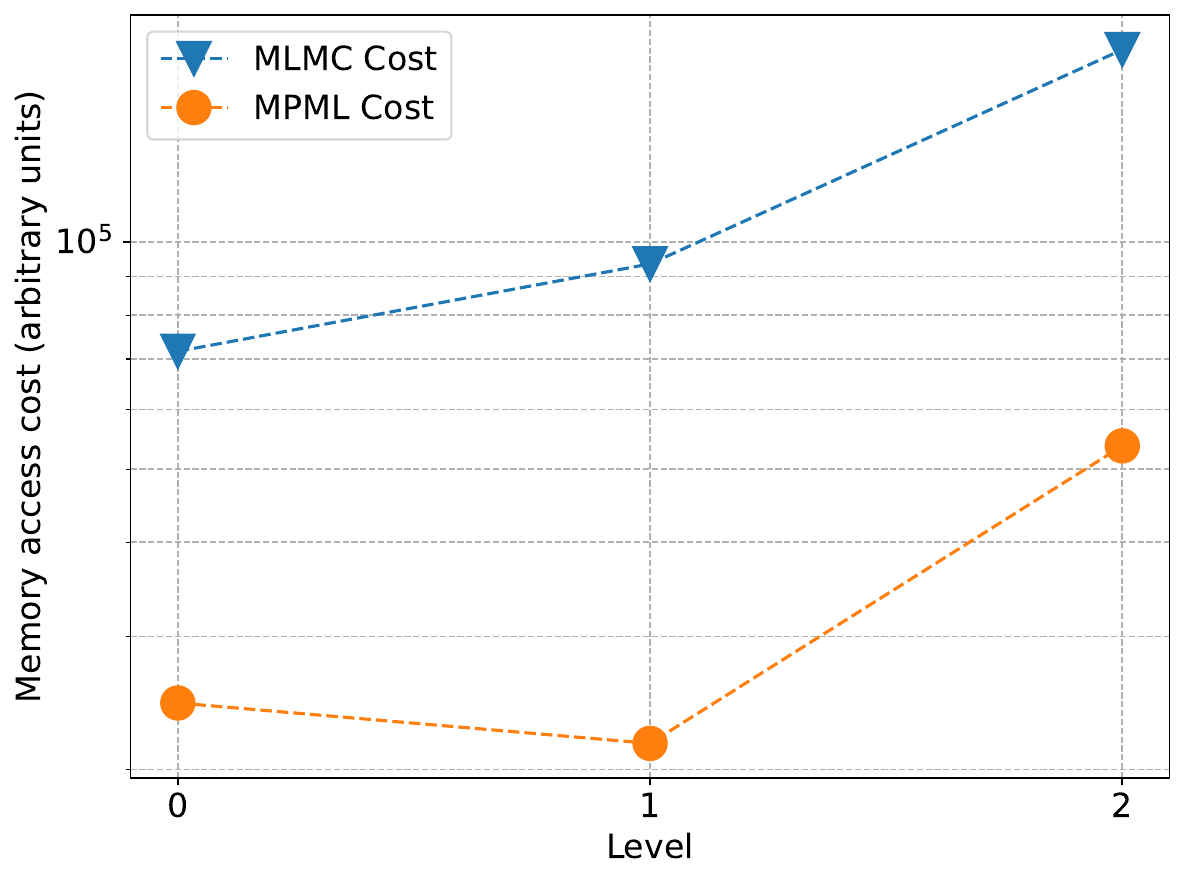}
        \captionof{figure}{Comparison of total costs of the estimators per level in terms of memory access for MPML  and MLMC. MLMC uses double precision Cholesky, while MPML uses low precision Cholesky with iterative refinement (here with quarter precision on level $0$).}     
        \label{fig_memory_cost_cholesky_simple}
\end{minipage}
\end{figure}

\subsection{\rev{Energy savings through iterative refinement}}
\label{sec:energy}

\rev{We obtained a simulated gain of up to $3.5\times$ in terms of memory references in our experiments, see Figure~\ref{fig_memory_cost_cholesky}. The ratio of the cost of one memory access to one floating point operation is continuing to grow in recent architectures \cite[Table~2]{jouppi2021tenlessons}. On modern 7nm semiconductor architectures, a cache (SRAM) memory access is on average $10\times$ to $100\times$ more expensive in terms of energy than a floating point operation, while
 accessing SRAM is between $10\times$ and $100\times$ cheaper than accessing DRAM. Therefore, an efficient parallel implementation of iterative refinement within the proposed adaptive algorithm (Algorithm~\ref{alg_MPMLMC}) that allows one to store more data in the cache memory can bring a significant reduction in energy cost, which is a promising area for future research.}

\section{Conclusion}
Multilevel sampling methods have proven to be powerful tools for uncertainty quantification, offering significant performance improvements by efficiently redistributing computational work across a hierarchy of models. In this paper, we demonstrated how leveraging computations of lower accuracy on coarser levels can further enhance the efficiency of these methods in high-performance computing applications. As a use case, we have developed an adaptive algorithm to determine the minimum required computational accuracy for each level in the multilevel Monte Carlo method.

Through two practical examples, we showcased the potential of our approach to obtain significant cost gains. Using a low-precision sparse direct solver with iterative refinement, we achieved a simulated memory gain of up to \(3.5\times\), while employing a MINRES iterative solver yielded a speedup of \(2\times\) in floating point operations. While these results highlight the significant potential for energy-aware scientific computing, \rev{the presented examples serve as a proof-of-concept for the developed method rather than a robust demonstration across a spectrum of problem difficulties.} The potential for future work lies in applying this approach to other uncertainty quantification frameworks, such as multilevel Markov chain Monte Carlo or the multilevel stochastic collocation method and exploring its broader applications in scientific computing. The efficient parallel implementation of low-precision solvers within multilevel sampling methods is of great interest and promises to bring further speedup.

\section*{Acknowledgements} This work is supported by the Carl Zeiss-Stiftung through the project “Model-Based AI: Physical Models and Deep Learning for Imaging and Cancer Treatment” and by the Deutsche Forschungsgemeinschaft (German Research Foundation) under Germany’s Excellence Strategy EXC 2181/1 - 390900948 (the Heidelberg STRUCTURES Excellence Cluster). The second author is supported by the Charles University Research Centre program No. UNCE/24/SCI/005 and the European Union (ERC, inEXASCALE, 101075632). Views and opinions expressed are those of the authors only and do not necessarily reflect those of the European Union or the European Research Council. Neither the European Union nor the granting authority can be held responsible for them. 


%% file: bibliography.bib
@article{hoeksema1985analysis,
  title={Analysis of the spatial structure of properties of selected aquifers},
  author={Hoeksema, Robert J and Kitanidis, Peter K},
  journal={Water Resources Research},
  volume={21},
  number={4},
  pages={563--572},
  year={1985},
  publisher={Wiley Online Library}
}

@article{freeze1975stochastic,
  title={A stochastic-conceptual analysis of one-dimensional groundwater flow in nonuniform homogeneous media},
  author={Freeze, R Allan},
  journal={Water Resources Research},
  volume={11},
  number={5},
  pages={725--741},
  year={1975},
  publisher={Wiley Online Library}
}

@book{scheidegger1957physics,
  title={The Physics of Flow through Porous Media},
  author={Scheidegger, Adrian E},
  year={1957},
  publisher={University of Toronto Press}
}

@article{giles2015multilevel,
  title={Multilevel {M}onte {C}arlo methods},
  author={Giles, Michael B},
  journal={Acta Numerica},
  volume={24},
  pages={259--328},
  year={2015},
  publisher={Cambridge University Press}
}

@article{nobile2008sparse,
  title={A sparse grid stochastic collocation method for partial differential equations with random input data},
  author={Nobile, Fabio and Tempone, Ra{\'u}l and Webster, Clayton G},
  journal={SIAM Journal on Numerical Analysis},
  volume={46},
  number={5},
  pages={2309--2345},
  year={2008},
  publisher={SIAM}
}

@article{cliffe2011multilevel,
  title={Multilevel {M}onte {C}arlo methods and applications to elliptic {PDE}s with random coefficients},
  author={Cliffe, K Andrew and Giles, Mike B and Scheichl, Robert and Teckentrup, Aretha L},
  journal={Computing and Visualization in Science},
  volume={14},
  pages={3--15},
  year={2011},
  publisher={Springer}
}

@book{higham2002accuracy,
  title={Accuracy and stability of numerical algorithms},
  author={Higham, Nicholas J},
  year={2002},
  publisher={SIAM}
}

@article{carson2017new,
  title={A new analysis of iterative refinement and its application to accurate solution of ill-conditioned sparse linear systems},
  author={Carson, Erin and Higham, Nicholas J},
  journal={SIAM Journal on Scientific Computing},
  volume={39},
  number={6},
  pages={A2834--A2856},
  year={2017},
  publisher={SIAM}
}

@phdthesis{vieuble2022mixed,
  title={Mixed precision iterative refinement for the solution of large sparse linear systems},
  author={Vieuble, Bastien},
  year={2022},
  school={INP Toulouse}
}

@inproceedings{brugger2014mixed,
  title={Mixed precision multilevel {M}onte {C}arlo on hybrid computing systems},
  author={Brugger, Christian and de Schryver, Christian and Wehn, Norbert and Omland, Steffen and Hefter, Mario and Ritter, Klaus and Kostiuk, Anton and Korn, Ralf},
  booktitle={2014 IEEE Conference on Computational Intelligence for Financial Engineering \& Economics (CIFEr)},
  pages={215--222},
  year={2014},
  organization={IEEE}
}

@book{ern2004theory,
  title={Theory and Practice of Finite Elements},
  author={Ern, Alexandre and Guermond, Jean-Luc},
  volume={159},
  year={2004},
  publisher={Springer}
}

@article{babuska2004galerkin,
  title={Galerkin finite element approximations of stochastic elliptic partial differential equations},
  author={Babuška, Ivo and Tempone, Ra{\'u}l and Zouraris, Georgios E},
  journal={SIAM Journal on Numerical Analysis},
  volume={42},
  number={2},
  pages={800--825},
  year={2004},
  publisher={SIAM}
}

@article{teckentrup2013further,
  title={Further analysis of multilevel {M}onte {C}arlo methods for elliptic {PDE}s with random coefficients},
  author={Teckentrup, Aretha L and Scheichl, Robert and Giles, Michael B and Ullmann, Elisabeth},
  journal={Numerische Mathematik},
  volume={125},
  pages={569--600},
  year={2013},
  publisher={Springer}
}

@article{carson2018accelerating,
  title={Accelerating the solution of linear systems by iterative refinement in three precisions},
  author={Carson, Erin and Higham, Nicholas J},
  journal={SIAM Journal on Scientific Computing},
  volume={40},
  number={2},
  pages={A817--A847},
  year={2018},
  publisher={SIAM}
}

@article{haji2016multi,
  title={Multi-index {M}onte {C}arlo: when sparsity meets sampling},
  author={Haji-Ali, Abdul-Lateef and Nobile, Fabio and Tempone, Ra{\'u}l},
  journal={Numerische Mathematik},
  volume={132},
  pages={767--806},
  year={2016},
  publisher={Springer}
}

@misc{nvidia_h100_specs,
  author       = {{NVIDIA Corporation}},
  howpublished = {URL: https://resources.nvidia.com/en-us-tensor-core},
  note         = {[Accessed 23-4-2023]},
  title        = {{NVIDIA} {H}100 {T}ensor {C}ore {GPU} {A}rchitecture},
  url          = {URL: https://resources.nvidia.com/en-us-tensor-core},
}

@article{higham2022mixed,
  title={Mixed precision algorithms in numerical linear algebra},
  author={Higham, Nicholas J and Mary, Theo},
  journal={Acta Numerica},
  volume={31},
  pages={347--414},
  year={2022},
  publisher={Cambridge University Press}
}

@article{higham2019squeezing,
  title={Squeezing a matrix into half precision, with an application to solving linear systems},
  author={Higham, Nicholas J and Pranesh, Srikara and Zounon, Mawussi},
  journal={SIAM Journal on Scientific Computing},
  volume={41},
  number={4},
  pages={A2536--A2551},
  year={2019},
  publisher={SIAM}
}

@article{giles2024rounding,
author = {Giles, Michael B. and Sheridan-Methven, Oliver},
title = {Rounding Error Using Low Precision Approximate Random Variables},
journal = {SIAM Journal on Scientific Computing},
volume = {46},
number = {4},
pages = {B502-B526},
year = {2024},
doi = {10.1137/23M1552814}
}

@inproceedings{horowitz20141,
  title={1.1 computing's energy problem (and what we can do about it)},
  author={Horowitz, Mark},
  booktitle={IEEE International Solid-State Circuits Conference (ISSCC), Digest of Technical Papers},
  pages={10--14},
  year={2014},
  organization={IEEE}
}

@article{dally2020domain,
  title={Domain-specific hardware accelerators},
  author={Dally, William J and Turakhia, Yatish and Han, Song},
  journal={Communications of the ACM},
  volume={63},
  number={7},
  pages={48--57},
  year={2020},
  publisher={ACM New York, NY, USA}
}

@inproceedings{jouppi2021tenlessons,
author = {Jouppi, Norman P. and Yoon, Doe Hyun and Ashcraft, Matthew and Gottscho, Mark and Jablin, Thomas B. and Kurian, George and Laudon, James and Li, Sheng and Ma, Peter and Ma, Xiaoyu and Norrie, Thomas and Patil, Nishant and Prasad, Sushma and Young, Cliff and Zhou, Zongwei and Patterson, David},
title = {Ten lessons from three generations shaped {G}oogle's {TPUv4i}},
year = {2021},
isbn = {9781450390866},
publisher = {IEEE Press},
doi = {10.1109/ISCA52012.2021.00010},
booktitle = {Proceed. 48th Annual International Symposium on Computer Architecture (ISCA '21)')},
pages = {1–14},
numpages = {14},
location = {Virtual Event, Spain},
OPTseries = {ISCA '21}
}

@inproceedings{haidar2018design,
  title={The design of fast and energy-efficient linear solvers: On the potential of half-precision arithmetic and iterative refinement techniques},
  author={Haidar, Azzam and Abdelfattah, Ahmad and Zounon, Mawussi and Wu, Panruo and Pranesh, Srikara and Tomov, Stanimire and Dongarra, Jack},
  booktitle={International Conference on Computational Science},
  pages={586--600},
  year={2018},
  organization={Springer}
}

@inproceedings{heinrich2001multilevel,
  title={Multilevel {M}onte {C}arlo methods},
  author={Heinrich, Stefan},
  booktitle={Large-Scale Scientific Computing: Third International Conference, LSSC 2001 Sozopol, Bulgaria, June 6--10, 2001 Revised Papers 3},
  pages={58--67},
  year={2001},
  organization={Springer}
}

@article{teckentrup2015multilevel,
  title={A multilevel stochastic collocation method for partial differential equations with random input data},
  author={Teckentrup, Aretha L and Jantsch, Peter and Webster, Clayton G and Gunzburger, Max},
  journal={SIAM/ASA Journal on Uncertainty Quantification},
  volume={3},
  number={1},
  pages={1046--1074},
  year={2015},
  publisher={SIAM}
}

@article{barth2011multi,
  title={Multi-level {M}onte {C}arlo finite element method for elliptic {PDEs} with stochastic coefficients},
  author={Barth, Andrea and Schwab, Christoph and Zollinger, Nathaniel},
  journal={Numerische Mathematik},
  volume={119},
  pages={123--161},
  year={2011},
  publisher={Springer}
}

@book{greenbaum1997iterative,
  title={Iterative Methods for Solving Linear Systems},
  author={Greenbaum, Anne},
  year={1997},
  publisher={SIAM}
}

@misc{baratta_2023_10447666,
  author       = {Baratta, Igor A. and
                  Dean, Joseph P. and
                  Dokken, Jørgen S. and
                  Habera, Michal and
                  Hale, Jack S. and
                  Richardson, Chris N. and
                  Rognes, Marie E. and
                  Scroggs, Matthew W. and
                  Sime, Nathan and
                  Wells, Garth N.},
  title        = {{DOLFINx}: The next generation {FEniCS} problem
                   solving environment
                  },
  month        = dec,
  year         = 2023,
  publisher    = {Zenodo},
  doi          = {10.5281/zenodo.10447666},
  url          = {https://doi.org/10.5281/zenodo.10447666},
  howpublished = {\url{https://doi.org/10.5281/zenodo.10447666}}
}

@Article{         petsc_4_py,
  title         = {Parallel distributed computing using {P}ython},
  author        = {Lisandro D. Dalcin and Rodrigo R. Paz and Pablo A. Kler and Alejandro Cosimo},
  journal       = {Advances in Water Resources},
  volume        = {34},
  number        = {9},
  pages         = {1124 - 1139},
  note          = {New Computational Methods and Software Tools},
  issn          = {0309-1708},
  doi           = {10.1016/j.advwatres.2011.04.013},
  year          = {2011}
}

@misc{pychop,
  author = {Xinye Chen},
  title = {pychop: A {P}ython package for simulating low-precision arithmetic},
  year = {2025},
  howpublished = {\url{https://pypi.org/project/pychop/}},
  note = {Accessed: 2025-01-30},
  url = {https://pypi.org/project/pychop/}
}

@article{dodwell2019multilevel,
  title={Multilevel {Markov chain Monte Carlo}},
  author={Dodwell, Tim J and Ketelsen, Christian and Scheichl, Robert and Teckentrup, Aretha L},
  journal={SIAM Review},
  volume={61},
  number={3},
  pages={509--545},
  year={2019},
  publisher={SIAM}
}
